\documentclass[a4paper,11pt,twoside]{amsart}

\usepackage{amsfonts,amsmath}
\usepackage{amssymb, latexsym}
\usepackage{amscd,amsthm}
\usepackage[all]{xy}
\usepackage{graphicx}
\usepackage{wrapfig}
\usepackage{pgf,tikz}
\usetikzlibrary{arrows,chains,
 decorations.pathreplacing,
shapes,%
}
\usepackage[colorlinks=true,citecolor=blue, urlcolor=blue, linkcolor=blue]{hyperref}
\usepackage{enumerate}

\newcommand{\marginlabel}[1]%
{\mbox{}\marginpar{\toggedleft\hspace{0pt}\bfseries\sf#1}}

\numberwithin{equation}{subsection}

\newtheorem{thm}{Theorem}[section]
\newtheorem{lem}[thm]{Lemma}
\newtheorem{prop}[thm]{Proposition}
\newtheorem{cor}[thm]{Corollary}

\newtheorem{cor-def}[thm]{Corollary--Definition}

\newtheorem*{thmInt}{Main Theorem}

\theoremstyle{definition}
\newtheorem{defn}[thm]{Definition}

\newtheorem{rem}[thm]{Remark}

\def\Stab{\mathop{\mathrm{Stab}}\nolimits}

\newcommand{\CC}{\mathbb{C}}

\renewcommand{\H}{\mathcal{H}}
\newcommand{\I}{\mathcal{I}}

\newcommand{\MMM}{\mathfrak{M}}

\renewcommand{\O}{\mathcal{O}}
\renewcommand{\P}{\mathcal{P}}
\newcommand{\PP}{\mathbb{P}}

\newcommand{\RR}{\mathbb{R}}

\newcommand{\ZZ}{\mathbb{Z}}
\renewcommand{\geq}{\geqslant}
\renewcommand{\leq}{\leqslant}

\newcommand{\cat}[1]{\begin{bf}#1\end{bf}}

\newcommand{\rra}{\twoheadrightarrow}

\newcommand{\abs}[1]{\left\vert#1\right\vert}
\newcommand{\set}[1]{\left\{#1\right\}}
\newcommand{\gen}[1]{\left\langle#1\right\rangle}

\def\llambda{\ensuremath{\boldsymbol{\lambda}}}

\DeclareMathOperator{\coh}{\cat{Coh}}
\DeclareMathOperator{\Hom}{Hom}
\DeclareMathOperator{\Ext}{Ext}

\DeclareMathOperator{\id}{id}

\DeclareMathOperator{\coker}{coker}
\DeclareMathOperator{\rk}{rk}

\DeclareMathOperator{\cExt}{\mathcal{E}\mathit{xt}}

\DeclareMathOperator{\cHom}{\mathcal{H}\mathit{om}}

\DeclareMathOperator{\ch}{ch}

\DeclareMathOperator{\codim}{codim}

\newcommand{\Db}{{\rm D}^{\rm b}}

\newcommand{\mJ}{\mathfrak{J}}
\newcommand{\mM}{\mathfrak{M}}

\newcommand{\dual}{\makebox[0mm]{}^{{\scriptstyle\vee}}}
\newcommand{\ddual}{\dual\dual}
\newcommand{\Gm}{C}
\newcommand{\Coh}{\mathbf{Coh}}

\newcommand{\on}[1]{\operatorname{#1}}

\begin{document}

\title[Generalized twisted cubics on a cubic fourfold]{Generalized twisted cubics on a cubic fourfold as a moduli space of stable objects}

\author[M.~Lahoz, M.~Lehn, E.~Macr\`i, and P.~Stellari]{Mart\'{\i} Lahoz, Manfred Lehn, Emanuele Macr\`i, and Paolo Stellari}

\address{M.L.: Institut de Math\'{e}matiques de Jussieu -- Paris Rive Gauche (UMR 7586), Universit\'{e} Paris Diderot -- Paris~7,
B\^{a}timent Sophie Germain, Case 7012, 75205 Paris Cedex 13, France}
\email{marti.lahoz@imj-prg.fr}
\urladdr{\url{http://webusers.imj-prg.fr/~marti.lahoz/}}

\address{M.L.: Institut f\"ur Mathematik, Johannes Gutenberg Universit\"at Mainz, 55099 Mainz, Germany}
\email{lehn@mathematik.uni-mainz.de}

\address{E.M.: Department of Mathematics, Northeastern University, 360 Huntington Avenue, Boston, MA 02115, USA}
\email{e.macri@northeastern.edu}
\urladdr{\url{https://web.northeastern.edu/emacri/}}

\address{P.S.: Dipartimento di Matematica ``F.~Enriques'', Universit{\`a} degli Studi di Milano, Via Cesare Saldini 50, 20133 Milano, Italy}
\email{paolo.stellari@unimi.it}
\urladdr{\url{http://users.unimi.it/stellari}}

\thanks{Mart\'i Lahoz is partially supported by the grant MTM2015-65361-P MINECO/FEDER, UE
and the grant number 230986 of the Research Council of Norway.
Emanuele Macr\`i is partially supported by the NSF grant DMS-1523496.
Paolo Stellari is partially supported by the grants FIRB 2012 ``Moduli Spaces and Their Applications'' and
the national research project ``Geometria delle Variet\`a Proiettive'' (PRIN 2010-11).}



\begin{abstract}
We revisit the work of Lehn--Lehn--Sorger--van Straten on twisted cubic curves in a cubic fourfold not containing a plane in terms of moduli spaces.
We show that the blow-up $Z'$ along the cubic of the irreducible holomorphic symplectic eightfold $Z$, described by the four authors, is isomorphic to an irreducible component of a moduli space of Gieseker stable torsion sheaves or rank three torsion free sheaves.

For a very general such cubic fourfold, we show that $Z$ is isomorphic to a connected component of a moduli space of tilt-stable objects in the derived category and to a moduli space of Bridgeland stable objects in the Kuznetsov component. 
Moreover, the contraction between $Z'$ and $Z$ is realized as a wall-crossing in tilt-stability. 

Finally, $Z$ is birational to an irreducible component 
of Gieseker stable aCM bundles of rank six.
\end{abstract}

\maketitle

\setcounter{tocdepth}{1}
\tableofcontents


\section*{Introduction}
The birational geometry of cubic fourfolds is a fascinating and challenging problem in algebraic geometry.
The guiding principle of this paper is to understand and reinterpret geometric constructions on cubic fourfolds in terms of sheaf theory and homological algebra.

The conjectural relation between the question of rationality of cubic fourfolds and their derived categories of coherent sheaves is now well-known, and it emerged in the work of Kuznetsov \cite{Kuz:4fold};
the derived category of a cubic fourfold has a semiorthogonal component, the Kuznetsov component, whose properties are supposed to detect rationality.
Addington and Thomas \cite{AT} showed that Kuznetsov's categorical approach to rationality essentially matches the more classical Hodge theoretical one due to Hassett \cite{Hassett}.

In this paper we deal with spaces of rational curves.
For low degrees, spaces of rational curves give rise to irreducible holomorphic symplectic (IHS) varieties.
Beauville and Donagi \cite{BD} showed that the Fano variety $F(Y)$ of lines on a cubic fourfold $Y$ is a smooth projective IHS variety of dimension four deformation equivalent to the Hilbert scheme of two points on a K3 surface.
More recently, by following geometric intuitions by Dolgachev and seminal works \cite{dJS,EPS,ES,PS}, Lehn, Lehn, Sorger, and van Straten \cite{LLSvS} studied the space of rational curves of degree three.
If the cubic fourfold $Y$ does not contain a plane, they proved that the irreducible component $M_3(Y)$ of the Hilbert scheme containing twisted cubic curves is a smooth projective variety of dimension ten.
The curves in $M_3(Y)$ always span a $\PP^3$, so there is a natural morphism from $M_3(Y)$ to the Grassmannian $\on{Grass}(3,\PP^5)$ of three-dimensional projective subspaces in $\PP^5$.
This morphism induces a fibration $M_3(Y) \to Z'(Y)$, which is a $\PP^2$-fiber bundle.
The variety $Z'(Y)$ is also smooth and projective of dimension eight.
Roughly speaking, $Z'(Y)$ is constructed as a moduli space of determinantal representations of cubic surfaces in $Y$ (see \cite{B1,dolgachev}, for more on determinantal representations).
Finally, in $Z'(Y)$ there is an effective divisor coming from non-CM twisted cubics on $Y$.
This divisor can be contracted, giving rise to a new variety denoted by $Z(Y)$.
The variety $Z(Y)$ is a smooth IHS variety of dimension eight.
It contains the cubic fourfold $Y$ and $Z'(Y)$ can be realized as the blow-up of $Z(Y)$ in $Y$.
In \cite{AL}, it was shown that $Z(Y)$ is deformation equivalent to a Hilbert scheme of four points on a K3 surface.

The goal of this article is to give an alternative construction of $Z(Y)$ and $Z'(Y)$ by building on the previous works \cite{BLMS,LMS,LMS1,MS1}.

\begin{thmInt}\label{thm:main}
Let $Y$ be a smooth cubic fourfold not containing a plane and let $H$ denote the class of a hyperplane section.
\begin{enumerate}[{\rm(1)}]
\item\label{item:main1} Let $\mathbf{v}_1=\left(0,0,H^2,0,-\frac{1}{4}H^4\right)$. 
Then $Z'(Y)$ is isomorphic to an irreducible component of the moduli space of Gieseker stable sheaves on $Y$ with Chern character $\mathbf{v}_1$.
\item Let $\mathbf{v}_2=\left(3,0,-H^2,0,\frac{1}{4}H^4\right)$. Then: 
\begin{enumerate}[{\rm(2a)}]
\item\label{item:main2a} $Z'(Y)$ is isomorphic to an irreducible component of the moduli space of Gieseker stable torsion free sheaves on $Y$ with Chern character $\mathbf{v}_2$.
\item\label{item:main2b} If $Y$ is very general, both $Z(Y)$ and $Z'(Y)$ are isomorphic to an irreducible component of the moduli space of tilt-stable objects on $\Db(Y)$ with Chern character $\mathbf{v}_2$.
The contraction $Z'(Y)\to Z(Y)$ is realized as a wall-crossing contraction in tilt-stability.
\item\label{item:main2c} If $Y$ is very general, then $Z(Y)$ is isomorphic to a moduli space of Bridgeland stable objects on $\cat{T}_Y$ with Chern character $\mathbf{v}_2$.
\end{enumerate}
\item\label{item:main3} Let $\mathbf{v}_3 = \left(6,-3H,-\frac{1}{2}H^2,\frac{1}{2}H^3,\frac{1}{8}H^4\right)$.
Then $Z(Y)$ is birational to a component of the moduli space of Gieseker stable aCM bundles on $Y$ with Chern character $\mathbf{v}_3$.
\end{enumerate}
\end{thmInt}

A cubic fourfold is \emph{very general} if the algebraic part $H^4(Y,\mathbb{Z})\cap H^{2,2}(Y)$ of the cohomology group $H^4(Y,\mathbb{Z})$ is the smallest possible, i.e., it is generated by the class of a smooth cubic surface.
Also, we denote by $\cat{T}_Y$ the Kuznetsov component of $Y$; this is the triangulated subcategory of $\Db(Y)$ defined by
\[
\begin{split}
\cat{T}_Y:=& \langle \O_Y, \O_Y(H), \O_Y(2H) \rangle^\perp \\
=&\set{G\in\Db(Y):\Hom^p_{\Db(Y)}(\O_Y(iH),G)=0,\text{ for all }p\text{ and }i=0,1,2}.
\end{split}
\]

Part~\eqref{item:main1} of the Main Theorem identifies $Z'(Y)$ with an irreducible component of the moduli space of ideal sheaves of generalized twisted cubics inside the corresponding cubic surface (see Proposition~\ref{prop:ideals}). 
Note that the moduli space of such ideals has more than one irreducible component (see Remark~\ref{rmk:seccomp1}).

Part~\eqref{item:main2a} is in some sense a reformulation of Part~\eqref{item:main1}: the rank three torsion free sheaves are obtained by mutation of the ideal sheaves in Part~\eqref{item:main1} (see Section~\ref{subsec:proj}).
A priori it is not clear why these sheaves are Gieseker stable, and this is the main content of this part of the theorem (see Proposition~\ref{prop:stabF}). 
The advantage of this description of $Z'(Y)$ is that the rank three torsion free sheaves associated to aCM twisted cubics
belong to the category $\cat{T}_Y$. As explained in \cite[Section~4]{KM}, this reconstructs the symplectic structure on the corresponding open subset.
The stability of the sheaves associated to aCM twisted cubics was already proved in \cite[Lemma~2.5]{SS}. 

Our motivation for Parts \eqref{item:main2b} and \eqref{item:main2c} is to directly construct $Z(Y)$ as a moduli space of objects in the derived category (or in $\cat{T}_Y$), thus reproving \cite{LLSvS}, as also suggested in \cite{AL}.
The main issues are the stability of the objects involved and the projectivity of the resulting moduli spaces.
We can solve the stability problem by restricting ourselves to very general cubic fourfolds; this is, anyway, an interesting case from many perspectives (see, for example, \cite{Hcub}).
We cannot yet solve the projectivity issue, and so to prove the Main Theorem we still have to rely on \cite{LLSvS}. On the other hand, we will observe later that the smoothness of the moduli spaces in Parts \eqref{item:main2b} and \eqref{item:main2c} is automatic since they parametrize objects contained in the K3 category $\cat{T}_Y$. In contrast to this, proving that $Z(Y)$ is smooth requires some work in \cite{LLSvS}.

Tilt-stability is an auxiliary notion of stability in the derived category, introduced in \cite{BMT}, as a direct generalization of Bridgeland stability on surfaces \cite{BK3}.
It depends on two real parameters, $\alpha$ and $\beta$, $\alpha>0$.
The basic fact is that when $\alpha$ is sufficiently large, the moduli space of tilt-stable objects with fixed numerical invariants is isomorphic to stable sheaves, where stability is now the usual notion of Gieseker stability (with Hilbert polynomial truncated at $\ch_2$).
Hence, for the Chern character $\mathbf{v}_2$ for $\alpha$ large and $\beta<0$, Part~\eqref{item:main2a} of the Main Theorem realizes $Z'(Y)$ as an irreducible component of the moduli space of tilt-stable objects.
The idea is now to vary $\alpha$ and study the transformations induced on the moduli space when stability changes (very much like usual variation of GIT quotients, see \cite{DH,Tha}).
In fact, Part~\eqref{item:main2b} arises by crossing the first wall: the induced map contracts other irreducible components, and induces on $Z'(Y)$ a blow-down onto $Z(Y)$.
As remarked before, the sheaves corresponding to aCM curves are in $\cat{T}_Y$.
Crossing the wall in tilt-stability is nothing but projecting the sheaves corresponding to non-CM curves onto the category $\cat{T}_Y$.
This wall-crossing interpretation for other constructions involving cubic hypersurfaces was already studied in our previous paper \cite{LMS1} for cubic threefolds; we also refer to \cite{S}, where wall-crossing techniques are treated more in detail in the case of the projective three-dimensional space.

Bridgeland stability conditions on the Kuznetsov component $\cat{T}_Y$ of a cubic fourfold have been constructed in \cite{BLMS}.
By using a similar argument as in \cite[Appendix~A]{BLMS}, we can prove that the objects in Part~\eqref{item:main2b} after crossing the wall are also Bridgeland stable, for very general cubic fourfolds. 
The advantage of working with Bridgeland stability is that the moduli spaces, if projective, are actually smooth connected IHS varieties. 
In fact, moduli spaces are expected for any K3 category to be proper symplectic algebraic spaces which are very close to be projective, since they are endowed with a natural non-trivial nef line bundle by \cite{BM:projectivity}.
Also, if Part~\eqref{item:main2c} could be extended to special cubic fourfolds, wall-crossing in Bridgeland stability would provide a wealth of IHS birational models for $Z(Y)$, similarly as in \cite{BM}.

Part~\eqref{item:main3} follows from Part~\eqref{item:main2a} via a second mutation, which is an autoequivalence of $\cat{T}_Y$ (see Section~\ref{subsec:ACM}). As before, the difficult part is to prove the stability of these vector bundles (see Proposition~\ref{prop:stabaCM}). 
It should also be observed that constructing families of stable aCM bundles is, in general, a difficult task.
The result above provides such a family in the rank six case.
This should be compared to the family of rank four stable aCM bundles exhibited in \cite{LMS}.

\subsection*{Plan of the paper}

The paper is organized as follows.
After some preliminaries about generalized twisted cubics and the construction of $Z(Y)$ (see Section~\ref{subsec:twisted}), we show that $Z'(Y)$ is isomorphic to (a component of) a moduli space of ideals (Part~\eqref{item:main1} of the Main Theorem; see Section~\ref{subsec:Z'ideals}).
By using this, we prove Part~\eqref{item:main2a} of the Main Theorem in Sections~\ref{subsec:proj} and \ref{subsec:proj2}.
We will also show that the moduli space of stable torsion free sheaves under consideration contains another irreducible component (see Section~\ref{subsect:secondfam}).
As a preparation, we recall in Section~\ref{subsec:Semiorth} the notion of semiorthogonal decomposition and Kuznetsov's description of the derived category of a cubic fourfold.

The proof of Part~\eqref{item:main2b} of the Main Theorem is carried out in Section~\ref{sec:tiltstable}.
This requires some preliminary results about tilt-stability discussed in Section~\ref{subsec:tilt}.
The wall-crossing argument discussed in Section~\ref{subsec:birtransf}, which concludes the proof of Part~\eqref{item:main2b}, needs a detailed analysis of the so-called first wall.
This is explained in Section~\ref{subsec:thewall}.

The proof of Part~\eqref{item:main2c} is carried out in Section~\ref{sec:Bridgelandstable}. A brief recall on Bridgeland stability on the Kuznetsov component of a very general cubic fourfold is in Section~\ref{subsec:BridgelandStability}.
The proof of the theorem is in Section~\ref{subsec:ProofOfPart2c}, after a few preliminary results in Section~\ref{subsec:BridgelandStableObjects}, which are a very mild generalization of some results in \cite[Appendix~A]{BLMS}.

To prove Part~\eqref{item:main3} of the Main Theorem, we need to move one step further and make another mutation.
More precisely,  Section~\ref{sect:acm} yields the desired aCM vector bundles.
In Section~\ref{subsec:Chern} we discuss some of their basic properties.
In the same section we describe a natural involution which is used in Section~\ref{subsec:stabMGamma} to prove their Gieseker stability.
Part~\eqref{item:main3} is finally proved in Section~\ref{subsec:proofthmmain1}.

\subsection*{Notation}

In this paper we work over the complex numbers.
For a smooth projective variety $X$, we denote by $\Db(X)$ the bounded derived category of coherent sheaves on $X$ and we refer to \cite{huy} for basics on derived categories.
We assume some familiarity with basic constructions and definitions about moduli spaces of stable bundles for which we refer to \cite{HL}.
For example, given a sheaf $F$ and an ample divisor $H$, we denote by $P(F,n):=\chi(F(nH))$ its Hilbert polynomial and by $p(F,n)$ its reduced Hilbert polynomial.

\section{The geometric setting}\label{sect:prel}

In this section we briefly recall the constructions in \cite{LLSvS} and we show that $Z'(Y)$ is isomorphic to a component of the moduli space of Gieseker stable sheaves containing the ideal sheaves of generalized twisted cubics inside the corresponding cubic surface.

\subsection{Generalized twisted cubics on cubic fourfolds}\label{subsec:twisted}

Let $Y$ be a smooth cubic fourfold not containing a plane.
Following \cite{LLSvS}, we denote by $M_3(Y):=\mathrm{Hilb}^{gtc}(Y)$ the irreducible component of the Hilbert scheme $\mathrm{Hilb}^{3n+1}(Y)$ containing the twisted cubics.
By \cite[Theorem~A]{LLSvS}, the moduli space $M_3(Y)$ is a smooth irreducible projective variety of dimension ten.

The curves $\Gm$ in $M_3(Y)$ are usually called \emph{generalized twisted cubics} and they can be divided into two classes depending on whether $\Gm$ is arithmetically Cohen-Macaulay (aCM) or non-Cohen-Macaulay (non-CM).
The latter ones are plane curves with an embedded point at a singular point of the curve.
The locus of non-CM curves is a Cartier divisor $J(Y)$ inside $M_3(Y)$.
Both aCM and non-CM curves span a $3$-dimensional linear subspace in $\PP^5$.

According to \cite{LLSvS}, the natural morphism
\[
s\colon M_3(Y)\to\mathrm{Grass}(3,\PP^5),
\]
sending a generalized twisted cubic $\Gm$ on $Y$ to the $3$-dimensional projective space $\langle\Gm\rangle$ in $\PP^5$ spanned by $\Gm$, factors through a smooth projective variety $Z'(Y)$
\begin{equation}\label{eqn:diagr}
\xymatrix{
M_3(Y)\ar[dd]_-{s}\ar[drr]^-{a}& & \\
& & Z'(Y)\ar[dll]^-{\pi}\\
\mathrm{Grass}(3,\PP^5).&&
}
\end{equation}
in such a way that $a:M_3(Y)\to Z'(Y)$ is a $\PP^2$-fiber bundle.
According to Section 3 in \cite{LLSvS}, a point $p\in \pi^{-1}([\PP^3])\subseteq Z'(Y)$ is given by the pair $([A],g)$ where $A$ is a $3\times 3$ matrix with linear entries and $g$ is an equation of the cubic surface $Y\cap \pi(p)$.
More precisely, $A$ is a stable matrix with respect to the reductive group $G = \mathrm{GL}_3 \times  \mathrm{GL}_3/\Delta_0$.
If $\det(A)\neq 0$, then the class $[A]$ is the orbit of $A$ with respect to $G$.
In that case $g=\det(A)$.
If $\det(A)=0$, then we can suppose that $A$ is skew-symmetric and the class $[A]$ is the orbit of $A$ inside the skew-symmetric matrices with respect to $\Gamma=\mathrm{GL}_3/\pm \id$ which acts via $\gamma\cdot A = \gamma A\gamma^t$.

Also, a point $q\in s^{-1}([\PP^3])\subseteq M_3(Y)$ is given by a pair $([A],g)$ where $A$ is a stable $3\times 3$ matrix with linear entries and $g$ is an equation of the cubic surface $Y\cap \pi(p)$, but the class of $[A]$ is different.
Indeed, if $q$ corresponds to an aCM curve, then $\det(A)=g$ and the class $[A]$ is the orbit of $A$ with respect to the parabolic subgroup $P=(\mathrm{GL}_3 \times P')/\CC^*\subseteq G$,
where $P'$ is the parabolic subgroup of elements that stabilize the subspace $\CC^2 \times \set{0} \subseteq \CC^3$.
Then $a(q)$ is given by taking further quotient and the fiber corresponds to $G/P\cong \PP^2$.

\medskip

Finally, by \cite[Theorem~B]{LLSvS}, the image of $a(J(Y))$ which is a Cartier divisor $D$ in $Z'(Y)$ can be contracted such that the contraction $Z(Y)$ is a smooth eight dimensional irreducible holomorphic symplectic manifold.
Hence, $Z'(Y)$ is the blow-up of $Z(Y)$ and the centre is isomorphic to $Y$ embedded as a Lagrangian submanifold in $Z(Y)$:
\[
\xymatrix{
Z'(Y) \ar[r]^-{b} & Z(Y)\\
D\cong\PP(T_Y) \ar[r]\ar@{^{(}->}[u] & Y \ar@{^{(}->}[u]^j
}
\]
where $b\colon Z'(Y)\to Z(Y)$ is the blow-up of $Z(Y)$ along $j(Y)$.

\subsection{$Z'$ as a moduli space of ideals}\label{subsec:Z'ideals}

We keep assuming that the cubic fourfold $Y$ is smooth and does not contain a plane.
Let $\mJ$ be the moduli space of Gieseker stable sheaves on $Y$ with reduced Hilbert polynomial
\begin{equation}\label{eqn:HP}
\frac{3}{2}n(n-1).
\end{equation}
Given a generalized twisted cubic $\Gm$ contained in the cubic surface $S\subseteq Y$, we have
\begin{equation}\label{eqn:calcHPI}
h^0(\I_{\Gm/S}(nH))=\frac32 n(n-1)\qquad\text{and}\qquad h^i(\I_{\Gm/S}(nH))=0,
\end{equation}
for $n\geq 1$.
Hence, the reduced Hilbert polynomial $p(\I_{\Gm/S},n)$ has the form \eqref{eqn:HP}.

\begin{lem}\label{lem:idstab}
The ideal sheaf $\I_{\Gm/S}$ on $Y$ is Gieseker stable for all generalized twisted cubics $\Gm$ in $Y$.
\end{lem}

\begin{proof}
The sheaf $\I_{\Gm/S}$ is torsion supported on the reduced and irreducible cubic surface $S$.
The result follows since  $\I_{\Gm/S}$ as a sheaf on $S$ is torsion free of rank one.
\end{proof}

For later use, let us remember the following natural isomorphisms
\begin{equation}\label{eqn:extideals}
\begin{split}
\Hom(\I_{\Gm/S},\I_{\Gm/S})&\cong H^0(S,\O_S)\cong\CC\\
\Ext^1(\I_{\Gm/S},\I_{\Gm/S})&\cong H^0(S,\O_S(H))^{\oplus 2}\cong\CC^8\\
\Ext^2(\I_{\Gm/S},\I_{\Gm/S})&\cong H^0(S,\O_S(2H))\cong\CC^{10},\\
\end{split}
\end{equation}
where $\Gm$ is an aCM generalized twisted cubic contained in the cubic surface $S\subseteq Y$.

\medskip

Let us now move towards the first description of $Z'(Y)$ as a moduli space of ideals; Proposition~\ref{prop:ideals} belows proves Part~\eqref{item:main1} of the Main Theorem. 

From the discussion in Section~\ref{subsec:twisted}, we know that $Z'(Y)$ parametrizes pairs $p=([A],g)$, where $A$ is a stable $3\times 3$ matrix with linear entries and $g$ is an equation of the cubic surface $S_p=Y\cap \pi(p)$.
As explained in the proof of Proposition~3.12 in \cite{LLSvS}, any choice of a two-dimensional subspace in the space generated by the column vectors of $A$ gives a $3\times 2$-matrix whose minors provide three quadrics generating the ideal $\I_p=(Q_1,Q_2,Q_3)\subseteq \O_{\PP^3}$.
Consider $\I'_p=\I_p+(g)\subseteq \O_{\PP^3}$ and take the quotient by $(g)$ such that we obtain $\I''_p=\I'_p/(g)\subseteq i_*\O_{S_p}$.
Note that if we take any curve $\Gm\in a^{-1}(p)\subseteq M_3(Y)$, we have
\[
\I_{\Gm/S_p}\cong \I''_p.
\]
By the discussion in Section 4 of \cite{LLSvS}, the previous assignment $p\mapsto \I''_p$ works in families over $\on{Grass}(3,\PP^5)$, giving a morphism
\[
f\colon Z'(Y)\longrightarrow\mJ
\]
which clearly factors through an irreducible component $\mJ_1$ of $\mJ$.

\begin{prop}\label{prop:ideals}
The morphism $f\colon Z'(Y)\to\mJ_1$ is an isomorphism.
\end{prop}

\begin{proof}
We denote by $\mJ'_1$ the image of $f$,  with the induced reduced structure, and we think of $f$ as a morphism from $Z'(Y)$ to $\mJ'_1$.
Set $k\colon\mJ'_1\hookrightarrow\mJ_1$ to be the inclusion.

The fact that $f$ is injective follows from the argument in the proof of Proposition~3.12 and the discussion in Section 3.1 of \cite{LLSvS}.
Indeed, a point in $\mJ'_1$ corresponds to an ideal $\I_{\Gm/S}$, where $\Gm$ is a generalized twisted cubic contained in the cubic surface $S$.
Thus, it determines uniquely the cubic equation $g$ cutting out $S$ in $\PP^3$.
The ideal of the twisted cubic inside $S$ can be given by three quadrics which are the minors of a $3\times 2$-matrix.
Depending on whether the twisted curve is aCM or non-CM, the matrix can be completed uniquely (up to the action of the corresponding group) to either a stable $3\times 3$-matrix whose determinant is $g$ or to a stable skew-symmetric matrix.

We want to prove that $f$ is actually a closed embedding.
For this, it is enough to show that it is injective on tangent spaces.
On the open complement $Z'(Y)\setminus D$ of the divisor of curves with embedded points this is a straightforward verification using long exact sequences of Ext-groups that we skip.
The situation is more delicate for points on the divisor $D$.
A curve $C$ corresponding to a point on $D$ is defined as subscheme of $Y$ by the following data:
\begin{itemize}
\item[(i)] The choice of a point $y\in Y$;
\item[(ii)] The choice of a $\PP^3$ passing through $y$
and contained in the projective tangent space of $Y$ at $y$, the intersection of which with $Y$ defines a
cubic surface $S$;
\item[(iii)] A plane $P$ in this $\PP^3$ passing through $y$.
\end{itemize}
The tangent space of $[C]\in M_3(Y)$ is ten-dimensional and spanned by the following first order deformations:
\begin{itemize}
\item[(i)] Four directions corresponding to infinitesimal translations of $y$ in $Y$;
\item[(ii)] Three directions corresponding to infinitesimal deformations of $\PP^3$ inside the projective tangent space of $Y$ at $y$;
\item[(iii)] Two directions corresponding to the changes of the choice of the plane $P$;
\item[(iv)] One direction that leads to the removal of the embedded point.
\end{itemize}
The two directions listed under (iii) are those contracted under the map $M_3(Y)\to Z'(Y)$.
The seven directions listed under (i) and (ii) effectively change the position of the 3-space $\PP^3\in \mathrm{Grass}(3,\PP^5)$.
The 7-dimensional subspace $\Theta\subset T_{[C]}Z'(Y)$ spanned by these directions therefore maps injectively into the tangent space $T_{[I_{C/S}]}\mJ_1$.
We need to focus on the deformation that removes the embedded point, i.e.\ that is transverse to $D\subset Z'(Y)$.
In order to facilitate the calculation we may choose coordinates $x_0,\ldots,x_5$ in $\PP^5$ in such a way that the embedded point is $y=[1:0:0:0:0:0]$, the tangent hyperplane
to $Y$ at $y$ is $\{x_5=0\}$, the three space spanned by $C$ is $\{x_4=x_5=0\}$ and the plane $P$ that contains the plane cubic curve $C_0\subset C$ equals $\{x_3=x_4=x_5=0\}$.
Then the cubic polynomial $f$ that defines $Y$ may be written in the form $f=\sum_{i,j=1}^3 g_{ij}x_ix_j+x_4q_4+x_5q_5$ with quadratic forms $q_4$ and $q_5$ and a symmetric matrix $(g_{ij})$ of linear forms in $x_0,x_1,x_2,x_3$.
The ideal $I_{C/Y}$ is generated by $x_4,x_5,x_1x_3, x_2x_3,x_3^2$.
Besides the tautological relations the last three generators have the following non-trivial syzygies:
\[\left(\begin{smallmatrix}x_1x_3&x_2x_3&x_3^2\end{smallmatrix}\right)\cdot
\left(\begin{smallmatrix}0&-x_3&x_2\\x_3&0&-x_1\\-x_2&x_1&0\end{smallmatrix}\right)=0.\]
A deformation transverse to $D$ is characterised by the property that the skew-symmetry of the $3\times 3$-matrix $A$ appearing in the equality above is destroyed (cf.\ the discussion in Section 3.3. of \cite{LLSvS}.).
The relevant deformation of $A$ is in fact given by $A\rightsquigarrow A+\varepsilon\left(g_{ij}\right)$.
The generators of $I_{C/Y}$ change to
\[(x_4,x_5,x_1x_3+\varepsilon g_2, x_2x_3-\varepsilon g_1, x_3^2).\]
Here $g_i=\sum_{j=1}^3 g_{ij}x_j$.
Note that in first order the linear forms $x_4$ and $x_5$ that define the 3-space $\langle C\rangle$ do not change.
The first order deformation under discussion is therefore exceptional for the projection $Z'(Y)\to \mathrm{Grass}(3,\PP^5)$.
For the same reason the corresponding deformation of $I_{C/S}$ will be linearly independent of the image of the seven dimensional space $\Theta$
discussed above, provided we can show that it is non-zero.
The assertion is therefore reduced to the task of showing that the extension $0\to I_{C/S}\to I'\to I_{C/S}\to 0$, where $I'\subset \O_S[\varepsilon]$ is generated by the quadrics $x_1x_3+\varepsilon g_2, x_2x_3+\varepsilon g_1, x_3^2$, is non-split.
Any splitting $s:I_{C/S}\to I'$ necessarily has the form $s(x_1x_3)=x_1x_3+\varepsilon(g_2+\gamma_1x_3)$, $s(x_2x_3)=x_2x_3+\varepsilon(-g_1+\gamma_2x_3)$, $s(x_3^2)=x_3^2+\varepsilon\gamma_3 x_3$ with $\gamma_i\in (x_1,x_2,x_3)$.
In order for $s$ to be well-defined the relation 
\[
\left(\begin{matrix}
0&x_3&-x_2\\
-x_3&0&x_1\\
x_2&-x_1&0
      \end{matrix}\right)
\left(\begin{matrix}
g_2+\gamma_1x_3\\-g_1+\gamma_2x_3\\\gamma_3 x_3
      \end{matrix}\right)=0\in \O_S^3\]
must hold.
Using the relation $x_1g_1+x_2g_2+x_3g_3=0$ and the fact that $S$ is integral, one gets
\[g_1=\gamma_2x_3-\gamma_3x_2,\quad g_2=\gamma_3x_1-\gamma_1x_3,\quad g_3=\gamma_1x_2-\gamma_2x_1.\]
If one writes $\gamma_i=\sum_{j=1}^3 \gamma_{ij}x_j$ with complex numbers $\gamma_{ij}$ and uses the relation $\frac{\partial g_i}{\partial x_j}=\frac{\partial g_j}{\partial x_i}$, one gets $\gamma_{ij}=\delta_{ij} \gamma_0$ for some $\gamma_0$, which immediately produces the contradiction $g_1=g_2=g_3=0$.

This means that the differential of $f$ is injective and $f$ is a closed embedding.
A general point in $\mJ_1$ is of the form $\I_{\Gm/S}$ for an aCM curve $\Gm$.
In this case, by \eqref{eqn:extideals},
\[
\dim\Ext^1(\I_{\Gm/S},\I_{\Gm/S})=\dim Z'(Y)=8.
\]
Therefore, the projective variety $Z'(Y)$ is embedded as a closed subvariety into an irreducible variety of the same dimension.
So $Z'(Y)$ is actually isomorphic to $\mJ_1$.
\end{proof}

\begin{rem}\label{rmk:seccomp1}
One could proceed further and describe a second irreducible component $\mJ_2$ in $\mJ$.
Instead of taking the ideal of a generalized twisted cubic $\Gm$ in a cubic surface $S\subseteq Y$, one can consider pairs $(p,S)$, where $p\in S\subseteq Y$ and $S$ is again a cubic surface.
This yields a bundle $G\to Y$ and the fiber over $p\in Y$ is the Grassmannian of $3$-planes in $\PP^5$ passing through $p$.
Thus, the fiber is isomorphic to $\on{Grass}(3,\PP^5)$ and $G$ has dimension ten.

Arguing as in the previous case, we can map a pair $(p,S)$ in $G$ to the corresponding ideal sheaf $\I_{p/S}(-H)$ in $Y$.
This gives a morphism $f'$ between $G$ and a second irreducible component $\mJ_2\subseteq\mJ$.
It is not difficult to see that $\mJ_1$ and $\mJ_2$ intersect each other.
This is indeed the image under $f$ of the divisor $D\subseteq Z'(Y)$, where all non-CM curves in $M_3(Y)$ are mapped by $a$.
Indeed, if $p\in D$, then $f(p)=\I_{\Gm/S}$ is such that $\I_{\Gm/S}\cong\I_{p/S}(-H)$, where $p$ is the singular point of the surface $S$ and $Y$ is tangent at $p$ to the $\PP^3$ containing $S$.

Moreover, one can check that, away from the intersection between $\mJ_1$ and $\mJ_2$, the morphism $f'$ is an isomorphism.
\end{rem}

\section{$Z'$ as a moduli space of Gieseker stable torsion free sheaves}\label{sect:F}

In this section, we show that the irreducible component $\mJ_1$ described in Section~\ref{subsec:Z'ideals} is isomorphic to an irreducible component of a moduli space of torsion free sheaves on the cubic fourfold.
All together, this proves Part~\eqref{item:main2a} of the Main Theorem.

As they will be used all along the paper, we first list some basic properties of semiorthogonal decompositions.
We focus on the derived categories of cubic fourfolds.

\subsection{Semiorthogonal decompositions and cubic fourfolds}\label{subsec:Semiorth}

Take a smooth projective variety $X$ and let $\Db(X)$ be its bounded derived category of coherent sheaves.
A \emph{semiorthogonal} decomposition of $\Db(X)$ is a sequence of full triangulated subcategories $\cat{T}_1,\ldots,\cat{T}_m\subseteq\Db(X)$ such that $\Hom_{\Db(X)}(\cat{T}_i,\cat{T}_j)=0$ for $i>j$ and, for all $G\in\Db(X)$, there exists a chain of morphisms in $\Db(X)$
\[
0=G_m\to G_{m-1}\to\ldots\to G_1\to G_0=G
\]
with $\mathrm{cone}(G_i\to G_{i-1})\in\cat{T}_i$ for all $i=1,\ldots,m$.
We will denote such a decomposition by $\Db(X)=\langle\cat{T}_1,\ldots,\cat{T}_m\rangle$.

An object $F\in\Db(X)$ is \emph{exceptional} if $\Hom_{\Db(X)}(F,F)\cong\CC$ and $\Hom_{\Db(X)}^p(F,F)=0$ for all $p\neq0$.
A collection $\{F_1,\ldots,F_m\}$ of objects in $\Db(X)$ is called an \emph{exceptional collection} if $F_i$ is an exceptional object for all $i$, and $\Hom_{\Db(X)}^p(F_i,F_j)=0$ for all $p$ and all $i>j$.

\begin{rem}\label{rmk:exceptional}
An exceptional collection $\{F_1,\ldots,F_m\}$ in $\Db(X)$ provides a semiorthogonal decomposition
\[
\Db(X)=\langle\cat{T},F_1,\ldots,F_m\rangle,
\]
where, by abuse of notation, we denoted by $F_i$ the triangulated subcategory generated by $F_i$ (equivalent to the bounded derived category of finite dimensional vector spaces).
Moreover
\[
\cat{T}:=\langle F_1,\ldots,F_m\rangle^\perp=\left\{G\in\Db(X)\,:\,\Hom^p(F_i,G)=0,\text{ for all }p\text{ and }i\right\}.
\]
Similarly, one can define ${}^\perp\langle F_1,\ldots,F_m\rangle=\left\{G\in\cat{T}\,:\,\Hom^p(G,F_i)=0,\text{ for all }p\text{ and }i\right\}$.
\end{rem}

For an exceptional object $F\in \Db(X)$, we consider the two functors, respectively \emph{left and right mutation}, $\cat{L}_F,\cat{R}_F:\Db(X)\to\Db(X)$ defined by
\begin{equation*}\label{eqn:LRmutation}
\begin{split}
\cat{L}_F(G)&:=\mathrm{cone}\left(\mathrm{ev}:\mathrm{RHom}(F,G)\otimes F\to G\right)\\
\cat{R}_F(G)&:=\mathrm{cone}\left(\mathrm{ev}^\vee:G\to\mathrm{RHom}(G,F)^\vee\otimes F\right)[-1],
\end{split}
\end{equation*}
where $\mathrm{RHom}(-,-):=\oplus_{p}\Hom_{\Db(X)}^p(-,-)[-p]$.

\smallskip

Let us now go back to the case of a cubic fourfold $Y$ in $\PP^5$.
As observed in \cite{Kuz:4fold}, we have a semiorthogonal decomposition
\begin{equation}\label{eqn:so1}
\Db(Y)=\langle\cat{T}_Y,\O_Y,\O_Y(H),\O_Y(2H)\rangle,
\end{equation}
where $H$ is a hyperplane section of $Y$.
The objects $\O_Y$, $\O_Y(H)$ and $\O_Y(2H)$ are exceptional and, by definition,
\begin{equation*}
\begin{split}
\cat{T}_Y&:=\langle\O_Y,\O_Y(H),\O_Y(2H)\rangle^\perp\\
&=\left\{G\in\Db(Y):\Hom^p_{\Db(Y)}(\O_Y(iH),G)=0,\text{ for all }p\text{ and }i=0,1,2\right\}.
\end{split}
\end{equation*}
Note that $\cat{T}_Y$ is a \emph{K3 category}: its Serre functor is the shift by $2$ and its cohomological properties are the same as those of $\Db(X)$, for $X$ a K3 surface.
By tensoring by $\O_Y(-H)$ the semiorthogonal decomposition \eqref{eqn:so1}, we have
\begin{equation}\label{eqn:so2}
\Db(Y)=\langle\cat{T}'_Y,\O_Y(-H),\O_Y,\O_Y(H)\rangle
\end{equation}
and $\cat{T}'_Y$ is naturally equivalent to $\cat{T}_Y$.

\subsection{The first mutation: general properties}\label{subsec:proj}

Assume, from now on, that $Y$ is a smooth cubic fourfold not containing a plane.
Take a generalized twisted cubic $\Gm$ in $M_3(Y)$ and fix the class $H$ of an ample divisor on $Y$.
Denote by $S$ the (reduced and irreducible) cubic surface in $Y$ containing $\Gm$.
From \eqref{eqn:calcHPI} we get
\begin{equation*}\label{eqn:cohid}
H^i(Y,\I_{\Gm/S}(2H))\cong\begin{cases}
0 & i\neq 0\\
\CC^3 & i= 0.
\end{cases}
\end{equation*}

The evaluation map
\[
\xymatrix{
H^0(Y,\I_{\Gm/S}(2H))\otimes\O_Y\ar[r]^-{\mathrm{ev_\Gm}}&\I_{\Gm/S}(2H)
}
\]
is surjective and we can then define the rank three torsion free sheaf
\[
F_\Gm:=\ker(\mathrm{ev}_\Gm)
\]
which sits in the short exact sequence
\begin{equation}\label{eqn:Fkerev}
\xymatrix{
0\ar[r]&F_\Gm\ar[r]&\O_Y^{\oplus 3}\ar[r]^-{\mathrm{ev}_\Gm}&\I_{\Gm/S}(2H)\ar[r]& 0.
}
\end{equation}
From this, we deduce that the Chern character of $F_\Gm$ is
\[
\ch(F_\Gm)=\left(3,0,-H^2,0,\frac{1}{4}H^4\right),
\]
which is precisely $\mathbf{v}_2$ of the Main Theorem, and its reduced Hilbert polynomial is
\begin{equation}\label{eqn:HilbPol}
p(F_\Gm,n)=\frac1{8}n^4+\frac{3}{4}n^3+\frac{11}{8}n^2+\frac{3}{4}n.
\end{equation}

\begin{rem}\label{rem:L1}
\begin{enumerate}[{\rm (i)}]
 \item By definition, $F_\Gm$ is actually obtained by applying the functor \[\cat{L}_{\O_Y}\left(-\otimes\O_Y(H) \right)[-1]\] to the sheaf $\I_{\Gm/S}(H)$.
 \item The notation $F_\Gm$ is partly misleading as $F_\Gm$ does not depend on $\Gm$ itself but on the ideal $\I_{\Gm/S}$.
Indeed, it was shown in \cite{AL} that if $\Gm_1$ and $\Gm_2$ are aCM generalized twisted cubics, then $b([\Gm_1])=b([\Gm_2])$ if and only if $F_{\Gm_1}\cong F_{\Gm_2}$.
We will use this fact at the end of Section~\ref{subsec:birtransf}.
\end{enumerate}
\end{rem}

We have the following result.

\begin{lem}\label{lem:FinTY}
We have $h^i(Y,F_\Gm)=0$ for all $i$ while if $\Gm$ is an aCM twisted cubic, then the sheaf $F_\Gm$ is in $\cat{T}_Y$.
\end{lem}

\begin{proof}
The fact that $h^i(Y,F_\Gm)=0$ for all $i$, is clear form Remark~\ref{rem:L1}.
On the other hand, $h^i(Y,F_\Gm(-H))=h^{i-1}(S,\I_{\Gm/S}(H))=0$ by \eqref{eqn:calcHPI}.
If $\Gm$ is aCM, then it has a resolution
in $\PP^3$ of the form (see, for example, \cite{B1,dolgachev})
\[
0\to \O_{\PP^3}(-3H)^{\oplus 3}\to \O_{\PP^3}(-2H)^{\oplus 3}\to \I_{\Gm/S}\to 0.
\]
Hence, $h^i(Y,F_\Gm(-2H))=h^{i-1}(S,\I_{\Gm/S})=0$ for all $i$.
\end{proof}

\begin{rem}\label{rmk:nonT}
If $C$ is a non-CM generalized twisted cubic on $Y$, we have
\[
\Hom^i(F_\Gm,\O_Y(-H))\cong \Hom^{4-i}(\O_Y(2H),F_\Gm)\cong\begin{cases}
\CC & i=1,2\\
0 &\text{otherwise}.
\end{cases}
\]
In particular, $F_\Gm$ is not an object of $\cat{T}_Y$, whenever $C$ is a non-CM generalized twisted cubic.
Moreover, combining this with Lemma~\ref{lem:FinTY}, we have that $\Hom(F_\Gm,\O_Y(-H)[1])$ is non-trivial if and only if $C$ a non-CM generalized twisted cubic on $Y$.
\end{rem}

\subsection{The first mutation: stability}\label{subsec:proj2}

First of all, given any generalized twisted cubic $\Gm$ in a cubic fourfold $Y$ not containing a plane, we can prove the following.

\begin{prop}\label{prop:stabF}
The sheaf $F_\Gm$ is Gieseker stable for all generalized twisted cubics $\Gm$ in $Y$.
\end{prop}

\begin{proof}
For sake of simplicity, let us just write $F$ for $F_\Gm$.
We need to show that the reduced Hilbert polynomial of any non-trivial proper saturated subsheaf $A\subset F$ satisfies $p(A,n)< p(F,n)$.
As $F$ has rank three, the subsheaf $A$ has rank one or rank two.

\smallskip

\noindent{\it Case $\rk(A)=1$.} As $A$ is torsion free, it has the form $A=\I_{W/Y}(m)$ for
some twist $m\in \ZZ$ and a subscheme $W\subset Y$ of codimension greater or equal than two.
The leading terms of $p(A,n)$ are $\frac{1}8n^4+(\frac m2+\frac 34)n^3+\ldots$.
Since $A$ is a subsheaf of $\O_Y^3$ as well, one has $m\leq 0$.
But if $m<0$, then $A$ is not destabilizing.
Hence, only the case $m=0$ and $A=\I_{W/Y}$ requires further consideration.
Since $h^0(Y,F)=0$, the subscheme $W\subset Y$ is non-empty.

Let $L$ denote the saturation of $\I_{W/Y}$ in $\O_Y^3$.
Then $L$ is a reflexive sheaf of rank one and hence invertible.
This shows that $L\cong\O_Y$.
We obtain a commutative diagram
\[
\xymatrix{
0\ar[r]&F\ar[r]&\O_Y^{\oplus 3}\ar[r]&\I_{\Gm/S}(2H)\ar[r]&0\\
0\ar[r]&\I_{W/Y}\ar[r]\ar[u]&\O_Y\ar[r]\ar[u]&\O_W\ar[r]\ar[u]&0
}
\]
of short exact sequences.
By assumption, $\I_{W/Y}$ is saturated in $F$ so that the quotient $F/\I_{W/Y}$ is torsion free and the map $\O_W\to \I_{\Gm/S}(2H)$ is injective.
Since $S$ is irreducible and reduced, this forces $W$ to contain $S$, so that
\[
P(\O_W,n)\geq P(\O_S,n)=\frac12(3n^2+3n+2)
\]
and
\[
P(\I_{W/Y},n)\leq \frac 18 n^4 + \frac 34 n^3 + \frac 38 n^2 + \frac 34 n< p(F,n)
\]
which completes the analysis in this case.

\smallskip

\noindent{\it Case $\rk(A)=2$.} Let $R\subset \O_Y^{\oplus 3}$ be the saturation of $A$ in $\O_Y^{\oplus 3}$.
The quotients $F/A$ and $\O_Y^{\oplus 3}/R$ are torsion free sheaves of rank one and therefore have the form $F/A\cong\I_{W/Y}(mH)$ and $\O_Y^{\oplus 3}/R\cong \I_{W'/Y}(m'H)$ for some integers $m,m'\in\ZZ$ and subschemes $W,W'\subset Y$ of codimension two.
As $F$ and $\O_Y^{\oplus 3}$ are isomorphic outside a codimension two locus, the same is true for $A$ and $R$ and for $\I_{W/Y}(mH)$ and $\I_{W'/Y}(m'H)$.
In particular, $m=m'$.
As $\I_{W'/Y}(mH)$ is globally generated, one has $m\geq 0$, and if $m>0$, the sheaf $\I_{W/Y}(mH)$ is not a destabilizing quotient of $F$.
This implies that $m=0$, and since $\I_{W'/Y}$ is globally generated, $W'=\emptyset$.
We obtain a commutative diagram
\[
\xymatrix{
&0&0&0&\\
0\ar[r]&\I_{W/Y}\ar[r]\ar[u]&\O_Y\ar[r]\ar[u]&\O_W\ar[r]\ar[u]&0\\
0\ar[r]&F\ar[r]\ar[u]&\O_Y^{\oplus 3}\ar[r]\ar[u]&\I_{\Gm/S}(2H)\ar[r]\ar[u]&0\\
0\ar[r]&A\ar[r]\ar[u]&\O_Y^{\oplus 2}\ar[r]\ar[u]&Q'\ar[r]\ar[u]&0\\
&0\ar[u]&0\ar[u]&0\ar[u]&\\
}
\]
with exact lines and columns.
Since $S$ is irreducible and reduced, the map $\O_Y^{\oplus 3}\to \I_{C/S}(2H)$ factors through $\O_S^{\oplus 3}$. This induces a surjection $\O_S\to \O_W$, namely $W$ must be a subscheme of $S$.
If $\I_{W/Y}$ is assumed to be a destabilizing quotient of $F$, we must have $p(\I_{W/Y},n)\leq p(F,n)$ or, equivalently, $P(\O_W,n)\geq p(\O_Y,n)-p(F,n)=\binom{n+2}{2}$.
This shows that $W$ is $2$-dimensional and hence equals $S$.
In this case, the support of the kernel $Q'$ is $1$-dimensional, which is impossible since $\I_{\Gm/S}$ is pure of dimension two.
Hence, we are done with the second case as well.
\end{proof}


\begin{proof}[Proof of Part~\eqref{item:main2a} of the Main Theorem]
By Proposition~\ref{prop:ideals}, we already know that $Z'(Y)$ is isomorphic to $\mJ_1$.
Thus, we just need to show that there is an irreducible component $\mM_1$ of the moduli space $\mM$ of stable sheaves with reduced Hilbert polynomial \eqref{eqn:HilbPol} isomorphic to $\mJ_1$.

By Remark~\ref{rem:L1}, the construction of $F_\Gm$ from $\I_{\Gm/S}(2H)$ is functorial and commutes with base change in flat families of generalized twisted cubics.
Thus, by Proposition~\ref{prop:stabF}, it defines a morphisms $f\colon\mJ_1\to\mM_1$, where $\mM_1$ is indeed an irreducible component of $\mM$.

Let us first prove that $f$ is bijective.
Indeed, since the support of the quotient $\O_Y^{\oplus 3}/F_\Gm$ has codimension two in $Y$, the inclusion $F_\Gm\to \O_Y^{\oplus 3}$ is isomorphic to the natural embedding $F_\Gm\to F_\Gm{\ddual}$.
In particular, $\I_{\Gm/S}(2H)$ can be reconstructed
from $F_\Gm$ as the quotient $F_\Gm\ddual/F_\Gm$.

Now we can prove that the differential $\mathrm{d}f$ of $f$ is injective.
Using the identifications of the tangent space of $\mJ_1$ in $\I_{\Gm/S}$ with $\Ext^1(\I_{\Gm/S},\I_{\Gm/S})$ and of $\mM$ at $F_\Gm$ with $\Ext^1(F_\Gm,F_\Gm)$, the differential $\mathrm{d}f$ is defined as follows.
Given $v\in\Ext^1(\I_{\Gm/S},\I_{\Gm/S})$, there exists $w\in\Ext^1(F_\Gm,F_\Gm)$ making the following diagram of distinguished triangles commutative
\[
\xymatrix{
F_\Gm\ar[r]\ar[d]^-{w}&\O_Y^{\oplus 3}\ar[d]^-{0}\ar[r]&\I_{\Gm/S}(2H)\ar[d]^-{v}\ar[r]&F_\Gm[1]\ar[d]^-{w[1]}\\
F_\Gm[1]\ar[r]&\O_Y^{\oplus 3}[1]\ar[r]&\I_{\Gm/S}(2H)[1]\ar[r]&F_\Gm[2].
}
\]
Indeed, the existence and uniqueness of $w$ is due to the fact that $\hom^i(\O_Y,F_\Gm)=h^i(Y,F_\Gm)=0$ for all $i$.
Hence, $\mathrm{d}f$ sends $v$ to $w$.
The injectivity of $\mathrm{d}f$ depends on the fact that $w$ is uniquely determined by $v$, as by Serre duality we have $\Ext^1(\I_{\Gm/S}(2H),\O_Y)\cong\Ext^3(\O_Y,\I_{\Gm/S}(-H))\dual$ and the latter space is actually trivial, being $\I_{\Gm/S}$ supported on the surface $S$.

So far we have that $f$ induces an isomorphism of $\mJ_1$ onto its image.
But if $\Gm$ is an aCM curve, then by Lemma~\ref{lem:FinTY}, the sheaf $F_\Gm$ is in $\cat{T}_Y$.
Thus, we have
\begin{equation*}
\Hom^i(F_\Gm,F_\Gm)\cong\Hom^{i+1}(\I_{\Gm/S}(2H),F_\Gm),
\end{equation*}
for $i=0,1,2$.
Moreover, by Serre duality, we get $\Hom(\I_{\Gm/S},\I_{\Gm/S})\cong\Ext^1(\I_{\Gm/S}(2H),F_\Gm)$ and the long exact sequence
\begin{equation}\label{eqn:les1}
\begin{split}
0\to &\Ext^1(\I_{\Gm/S},\I_{\Gm/S})\to \Ext^2(\I_{\Gm/S}(2H),F_\Gm)\to \Ext^2(\I_{\Gm/S}(2H),\O_Y^{\oplus 3})\to\\
\to &\Ext^2(\I_{\Gm/S},\I_{\Gm/S})\to \Ext^3(\I_{\Gm/S}(2H),F_\Gm)\to 0.
\end{split}
\end{equation}
On the one hand, we have the natural isomorphisms
\[
\Ext^i(\I_{\Gm/S}(2H),\O_Y)\cong H^{4-i}(Y,\I_{\Gm/S}(-H))\cong
\begin{cases}
\CC^3 &\text{if } i=2\\ 0&\text{otherwise}
\end{cases}
\]
and we have already computed \eqref{eqn:extideals}.
Using again that $F_\Gm\in \cat{T}_Y$, we have
\[
\Ext^2(F_\Gm,F_\Gm)\cong\Hom(F_\Gm,F_\Gm)\cong\Hom(\I_{\Gm/S},\I_{\Gm/S})\cong\CC.
\]
Therefore, the long exact sequence \eqref{eqn:les1} becomes:
\begin{equation*}
0\to \CC^8\to \Ext^2(\I_{\Gm/S}(2H),F_\Gm)\to \CC^9\to\CC^{10}\to \CC\to 0.
\end{equation*}

In conclusion $f$ induces an isomorphism $\Ext^1(F_\Gm,F_\Gm)\cong\Ext^2(\I_{\Gm/S}(2H),F_\Gm)\cong\CC^8$ and thus $\mathrm{d}f$ is an isomorphism as well.
Hence, $f$ induces an isomorphism between $\mJ_1$ and the irreducible component $\mM_1$.
\end{proof}

\subsection{The second family}\label{subsect:secondfam}

We now want to show that $\mM$ contains at least another irreducible component $\mM_2$.
The discussion here goes along the same lines as in Sections~\ref{subsec:proj} and \ref{subsec:proj2}.
Thus, we will be a bit quicker explaining the arguments.

We pick a point $p\in Y$ and a linear $3$-dimensional subspace $U\subset \PP^5$ that passes through $p$.
Then  $\I_{p/S}(H)$ has exactly three linearly independent global sections, the evaluation map
\[
\xymatrix{
H^0(Y,\I_{p/S}(H))\otimes\O_Y\ar[r]^-{\mathrm{ev_p}}&\I_{p/S}(H)
}
\]
is surjective and we get a rank three torsion free sheaf $E_p:=\ker(\mathrm{ev}_p)$.
Again, $E_p$ is obtained by applying the functor $\cat{L}_{\O_Y}\left(-\otimes\O_Y(H) \right)[-1]$ to the sheaf $\I_{p/S}$.
The reduced Hilbert polynomial of $E_p$ is the same as in \eqref{eqn:HilbPol}.
Moreover we have the following.

\begin{prop}\label{prop:stabE}
The sheaf $E_p$ is Gieseker stable for all $p\in Y$.
\end{prop}

\begin{proof}
The argument is exactly the same as in Proposition~\ref{prop:stabF} substituting $\I_{\Gm/S}(2H)$ by $\I_{p/S}(H)$.
Thus, we leave the easy check to the reader.
\end{proof}

Varying the point $p\in Y$ and the $3$-dimensional projective space containing it produces another sheaves contained in another irreducible component $\mM_2$ inside $\mM$.
More precisely,  the irreducible component $\mJ_2$ described in Remark~\ref{rmk:seccomp1} injects into $\mM_2$ of $\mM$.
Here, the argument is very similar to the one in the previous Section~\ref{subsec:proj2}.
Indeed, the procedure that associates $E_p$ to $\I_{p/S}(H)$ described above yields a morphism $g\colon\mJ_2\to\mM_2$ which is bijective onto its image because, again, the ideal $\I_{p/S}(H)$ can be reconstructed from $E_p$ as the quotient $E_p\ddual/E_p$.

\section{{aCM} twisted cubics and aCM bundles}\label{sect:acm}

In this section we associate an aCM bundle to an aCM curve in a cubic fourfold $Y$ not containing a plane.
For this, we need some general results in \cite{LMS1} which we recall in Section~\ref{subsec:ACM}.
The Gieseker stability of this aCM bundle will be discussed in Section~\ref{subsec:Chern}.

\subsection{{aCM} bundles on cubics}\label{subsec:ACM}

Let us briefly summarize some general results from \cite{LMS1} which have a sort of general flavour and apply to any smooth cubic hypersurface.
To begin with, consider the following.

\begin{defn}\label{def:ACMbalanced}
A vector bundle $F$ on a smooth projective variety $X$ of dimension $n$ is \emph{arithmetically Cohen-Macaulay} (aCM) if $\dim H^i(X,F(jH))=0$ for all $i=1,\ldots,n-1$ and all $j\in\ZZ$.
\end{defn}

The existence of families of (stable) aCM bundles is in general related to the so-called representation type of a variety.
It is in general not that easy to produce such families.
An example of a two dimensional family of (Gieseker) stable aCM vector bundles on a cubic fourfold containing a plane was exhibited in \cite[Theorem~A]{LMS}.
To get such a result we used a simple criterion that we recall here below.

\begin{lem}[{\cite[Lemma~1.9]{LMS1}}]\label{lem:viceversa}
Let $Y\subset \PP^{n+1}$ be a smooth cubic $n$-fold and let $F\in \coh(Y) \cap \cat{T}_Y$.
Assume
\begin{equation}\label{eqn:vanishing1}
\begin{split}
& H^1(Y,F(H))=0\\
& H^1(Y,F((1-n)H))=\ldots=H^{n-1}(Y,F((1-n)H))=0.
\end{split}
\end{equation}
Then $F$ is an aCM bundle.
\end{lem}

The idea is that, in presence of a cubic hypersurface, one can show that a sheaf is an aCM vector bundle just by proving much less cohomology vanishings.

\subsection{The second mutation: producing aCM bundles in $\mathbf{T}_Y$}\label{subsec:acmcurveacmbundle}

Let $\Gm$ be an aCM twisted cubic in $Y$ and let $S=\gen{\Gm}\cap Y$ the cubic surface containing $\Gm$.
Consider
\begin{equation*}
M_\Gm:=\cat{L}_{\O_Y}\left(\cat{L}_{\O_Y}\left(\I_{\Gm/S}(2H) \right)\otimes \O_Y(H)\right)[-2].
\end{equation*}
We also set
\begin{equation}\label{eqn:laGC} G_\Gm:=\ker\left(\O_S(H)^{\oplus 3}\xrightarrow{\on{ev}} \I_{\Gm/S}(3H)\right).\end{equation}

\begin{rem}\label{rmk:autoeq}
It is not difficult to see that the functor $\cat{L}_{\O_Y}\left(-\otimes \O_Y(H)\right)[-1]$ is an autoequivalence of $\cat{T}_Y$.
\end{rem}

\begin{lem}\label{lem:Mfeix}
The object $M_\Gm$ is a rank 6 aCM bundle in $\cat{T}_Y$.
\end{lem}

\begin{proof}
By Lemma~\ref{lem:FinTY} and Remark~\ref{rem:L1}, $F_\Gm=\cat{L}_{\O_Y}\left(\I_{\Gm/S}(2H) \right)[-1]$ is a sheaf in $\cat{T}_Y$.
Hence, $M_\Gm\in \cat{T}_Y$, for example by \cite[Lemma~1.10]{LMS}.
To prove that $M_\Gm$ is a sheaf we need to prove that $F_\Gm(H)$ is globally generated, so that $M_\Gm=\ker \left(\mathrm{ev}_F: \O_Y^{\oplus 9}\to F_\Gm(H)\right)$.

Since the evaluation map $\on{ev}:\O_Y(H)^{\oplus 3}\to \I_{\Gm/S}(3H)$ factors through $\on{ev}:\O_S(H)^{\oplus 3}\to \I_{\Gm/S}(3H)$,
there is a natural injection $\I_{S/Y}(H)^{\oplus 3}\hookrightarrow F_\Gm(H)$ whose cokernel is $G_\Gm$.
Hence, to prove that $F_\Gm(H)$ is globally generated, it is enough to prove that $G_\Gm$ is.

Note that we have
\[
0\to \O_{\PP^3}(-3H)^{\oplus 3}\to \O_{\PP^3}(-2H)^{\oplus 3}\to \I_{\Gm/S}\to 0.
\]
By tensoring this exact sequence by $\O_S(3H)$, we obtain the exact sequence
\[
\xymatrix{
\O_S^{\oplus 3}\ar[r]&\O_S(H)^{\oplus 3}\ar[r]^-{\on{ev}}&\I_{\Gm/S}(3H)\ar[r]& 0,
}
\]
where $G_{\Gm}=\ker(\on{ev})$.
Hence, $G_\Gm$ is globally generated and we have proved that $M_\Gm$ is a sheaf.

To prove that $M_\Gm$ is an aCM bundle we use Lemma~\ref{lem:viceversa}.
Since $M_\Gm$ is the kernel of an evaluation map and $H^1(Y,\O_Y(H))=0$, it follows that $H^1(Y,M_\Gm(H))=0$.
Moreover, since $H^i(Y,F_\Gm(-2H))=0$ for all $i$ (see Lemma~\ref{lem:FinTY}), we have that $H^i(Y,M_{\Gm}(-3H))=0$ for $i=1,2,3$.
So all the vanishings required in Lemma~\ref{lem:viceversa} hold true and we are done.
\end{proof}

\begin{rem}\label{rmk:compactMC}
If $\Gm$ is a non-CM twisted cubic in $Y$, then one can still consider the object $M_\Gm$ formally defined as above but in this case $M_\Gm$ is not a coherent sheaf but an actual complex of coherent sheaves.
\end{rem}

\subsection{Some properties of $M_\Gm$}\label{subsec:Chern}

To compute the Chern character of $M_\Gm$ we need the following preliminary result.
It will be used to control the stability of $M_\Gm$ as well.

\begin{lem}\label{lem:asinLMS1}
Let $\Gm$ be an aCM twisted cubic in $Y$ and $S=\gen{\Gm}\cap Y$.
Then the sheaf $M_\Gm$ sits in the following non-split exact sequence
\begin{equation}\label{eqn:laM}
0\to \O_{Y}(-H)^{\oplus 3}\to M_{\Gm}\to K_{\Gm}\to 0,
\end{equation}
where $K_\Gm:=\ker\left(\O_Y^{\oplus 3}\xrightarrow{\on{ev}} G_{\Gm}\right)$ sits in an exact sequence
\begin{equation}\label{eqn:laKGm}
0\to \I_{S/Y}^{\oplus 2}\to K_\Gm\to \I_{\Gm/Y}\to 0.
\end{equation}
\end{lem}

\begin{proof}
By definition $M_\Gm$ sits inside the commutative diagram
\begin{equation*}
\xymatrix{
&&0\ar[d]&0\ar[d]\\
&&M_{\Gm}\ar[d]\ar[r]&K_{\Gm}\ar[d]\ar[r]&0\\
0\ar[r]&\O_{Y}^{\oplus 6}\ar[d]\ar[r]&\O_{Y}^{\oplus 9}\ar[d]^{\on{ev}}\ar[r]&\O_{Y}^{\oplus 3}\ar[d]^{\on{ev}}\ar[r]&0\\
0\ar[r]&\I_{S/Y}(H)^{\oplus 3}\ar[d]\ar[r]&F_\Gm(H)\ar[d]\ar[r]^-{\alpha}&G_\Gm\ar[d]\ar[r]&0,\\
&0&0&0}
\end{equation*}
with exact rows and columns. Here $\alpha$ is the map completing the commutative diagram
\begin{equation*}
\xymatrix{
0\ar[r]& F_\Gm(H)\ar[r]&\O_Y(H)^{\oplus 3}\ar[r]\ar[d]^-{\oplus\beta}&\I_{\Gm/S}(3H)\ar@{=}[d]\ar[r]&0\\
0\ar[r]& G_\Gm\ar[r]&\O_S(H)^{\oplus 3}\ar[r]&\I_{\Gm/S}(3H)\ar[r]&0
}
\end{equation*}
and $\beta\colon\O_Y(H)\to\O_S(H)$ is the restriction map.
In particular, the kernel of $\alpha$ is $\I_{S/Y}(H)^{\oplus 3}$.
The kernel of the evaluation map $\on{ev}:\O_Y^{\oplus 6}\to \I_{S/Y}(H)^{\oplus 3}$ is $\O_Y(-H)^{\oplus 3}$.
Hence, we get the first part of the statement.

To get \eqref{eqn:laKGm}, one argues as follows.
If $S$ is an integral cubic surface, by \cite[Lemma~2.5]{CH1}\footnote{More precisely, it follows from Eisenbud's equivalence between matrix factorizations and mCM-modules.},
then $G_\Gm$ also sits in the following exact sequence on $S$
\[
\xymatrix{
0\ar[r]&\I_{\Gm/S}\ar[r]&\O_S^{\oplus 3}\ar[r]^-{\on{ev}}&G_\Gm\ar[r]& 0.
}
\]
Thus, since the evaluation map $\on{ev}:\O_Y^{\oplus 3}\to G_{\Gm}$ factors through $\on{ev}:\O_S^{\oplus 3}\to G_{\Gm}$,
there is an injection $\I_{S/Y}(H)^{\oplus 3}\hookrightarrow K_\Gm$ whose cokernel is $\I_{\Gm/S}$, i.e., we have an exact sequence
\begin{equation}\label{eqn:laK3I}
0\to \I_{S/Y}^{\oplus 3}\to K_\Gm \to \I_{\Gm/S}\to 0.
\end{equation}
Now, consider the commutative diagram
\begin{equation*}
\xymatrix{
&0\ar[d]&0\ar[d]\\
&\I_{S/Y}^{\oplus 2}\ar[d]\ar@{=}[r]&\I_{S/Y}^{\oplus 2}\ar[d]\\
0\ar[r]&\I_{S/Y}^{\oplus 3}\ar[d]\ar[r]^{\on{ev}}&K_\Gm\ar[d]\ar[r]&\I_{\Gm/S}\ar@{=}[d]\ar[r]&0\\
0\ar[r]&\I_{S/Y}\ar[d]\ar[r]&U\ar[d]\ar[r]&\I_{\Gm/S}\ar[d]\ar[r]&0.\\
&0&0&0}
\end{equation*}
with exact rows and columns. The vertical map $\I_{S/Y}^{\oplus 2}\to\I_{S/Y}^{\oplus 3}$ is any splitting inclusion. 
A simple calculation shows that $\Ext^1(\I_{\Gm/S},\I_{S/Y})\cong\CC$. Thus the second part of the statement follows once we prove that $U$ does not split as a direct sum of $\I_{\Gm/S}$ and $\I_{S/Y}$.

To this extent, it is enough to show that $\Hom(K_\Gm,\I_{S/Y})=0$.
For this, consider the short exact sequence
\[
0\to \O_Y(-2H)\to\O_Y(-H)^{\oplus 2}\to \I_{S/Y}\to 0.
\]
Applying $\Hom(K_\Gm,-)$ to it we get
\[
\Hom(K_\Gm,\O_{Y}(-H))^{\oplus 2}\to\Hom(K_\Gm,\I_{S/Y})\to\Ext^1(K_\Gm,\O_Y(-2H)).
\]
Since $\Hom( \I_{\Gm/S},\O_Y(-H))\cong\Hom(\I_{S/Y}^{\oplus 3},\O_Y(-H))=0$, we get $\Hom(K_\Gm,\O_{Y}(-H))=0$. 
On the other hand, by the definition of $K_\Gm$, we get $\Ext^1(K_C,\O_Y(-2H))\cong\Ext^2(G_C,\O_Y(-2H))$, and so $\Ext^2(G_C,\O_Y(-2H))\cong H^2(Y,G_\Gm(-H))\cong H^1(S,\I_{\Gm/S}(2H))=0$ by \eqref{eqn:calcHPI}.
This concludes the proof.
\end{proof}

By the Grothendieck--Riemann--Roch theorem, we can compute the following Chern characters:
\begin{equation*}\label{eqn:cherns}
\begin{split}
\ch(\I_{S/Y})&=\left(1,0,-H^2,3l,-\tfrac{7}{4}pt\right) \\
\ch(K_{\Gm})&=\left(3,0,-2H^2,3l,0\right) \\
\ch(\O_Y(-H)^{\oplus 3})&=\left(3,-3H,\tfrac{3}{2}H^2,-\tfrac{3}{2}l,\tfrac{3}{8}pt\right).
\end{split}
\end{equation*}
Hence, by applying Lemma~\ref{lem:asinLMS1}, we deduce
\begin{equation}\label{eqn:chM}
\ch(M_{\Gm})=\left(6,-3H,-\tfrac{1}{2}H^2,\tfrac{3}{2}l,\tfrac{3}{8}pt\right),
\end{equation}
which is exactly $\mathbf{v}_3$ from Part \eqref{item:main3} of the Main Theorem.
The reduced Hilbert polynomial of $M_C$ is
\begin{equation}\label{eqn:HilbertPolyMC}
p(M_\Gm,n):=\frac1{8}n^4+\frac{1}{2}n^3+\frac{5}{8}n^2+\frac{1}{4}n.
\end{equation}

The following lemma will be used later on and provides a natural involution for the aCM bundles $M_\Gm$.

\begin{lem}\label{lem:duality}
Let $\Gm$ be an aCM twisted cubic in $Y$ and set $S=\gen{\Gm}\cap Y$.
The sheaf $M'_\Gm:= \cHom(M_\Gm,\O_Y(-H))$ is naturally isomorphic to $M_{\Gm'}$ for some aCM twisted cubic $\Gm'\subset S$.
\end{lem}

\begin{proof}
Applying the functor $\cHom(-,\O_Y(-H))$ to the exact sequence \eqref{eqn:laM}, we obtain
\[
0\to \O_{Y}(-H)^{\oplus 3}\to M_{\Gm}^\vee \otimes \O_Y(-H) \to \O_Y^{\oplus 3}\to \cExt^1(K_{\Gm},\O_Y(-H))\to 0.
\]
By \eqref{eqn:laK3I}, we have
\begin{equation}\label{eqn:auxi}
0\to \cExt^1(K_\Gm,\O_Y(-H))\to \cExt^1(\I_{S/Y}^{\oplus 3},\O_Y(-H))\to \cExt^2(\I_{\Gm/S},\O_Y(-H))\to 0,
\end{equation}
where we recall that, by the Koszul resolution of $\I_{S/Y}$ or $\O_S$, $\cExt^1(\I_{S/Y},\O_Y)\cong \cExt^2(\O_{S},\O_Y) \cong \O_S(2H)$.
Hence, we have $\cExt^1(\I_{S/Y}^{\oplus 3},\O_Y(-H))\cong \O_S(H)^{\oplus 3}$ and, since $\I_{\Gm/S}$ is a line bundle $L_S$ on $S$,
\[
\cExt^2(\I_{\Gm/S},\O_Y(-H))\cong L^{-1}_S(H).
\]
Recall that we have
\[
0\to \O_{\PP^3}(-3H)^{\oplus 3}\to \O_{\PP^3}(-2H)^{\oplus 3}\to L_S\to 0.
\]
Hence
\[
0\to \O_{\PP^3}(2H)^{\oplus 3}\to \O_{\PP^3}(3H)^{\oplus 3}\to L_S^{-1}(3H) \to 0.
\]
Thus,
\begin{equation}\label{eqn:I_1Gm/S}
L_S\dual\cong\I_{\Gm'/S}(2H)
\end{equation}
for some other twisted cubic $C'$ inside $S$.
Summing up, from \eqref{eqn:auxi} we obtain
\[
0\to \cExt^1(K_\Gm,\O_Y(-H))\to \O_{S}(H)^{\oplus 3}\stackrel{\on{ev}}\longrightarrow \I_{\Gm'/S}(3H)\to 0.
\]
Thus, $\cExt^1(K_\Gm,\O_Y(-H))\cong G_{\Gm'}$ and as in  Lemma~\ref{lem:asinLMS1}, we obtain the desired presentation of $M'_{\Gm}$.
\end{proof}

\begin{rem}\label{rmk:Do}
Given $\Gm$ a twisted cubic in $S$, then, for example by \cite[Proposition~6.2]{B1} or \cite[Theorem~4.2.22]{dolgachev},
we have that $\O_S(\Gm)\cong\I_{\Gm/S}\dual\cong\I_{\Gm'/S}$ corresponds to a determinantal representation of $S$.
Here $\Gm'$ is the aCM twisted cubic in Lemma~\ref{lem:duality}.
There are 72 determinantal representation of $S$.
Moreover, we can associate to $\O_{S}(\Gm)$ the morphism
\[
S\to \abs{\O_{S}(\Gm)}^\vee=\abs{\I_{\Gm'/S}(2H)}^\vee
\]
and the latter linear system is isomorphic to $\PP^2$.
This gives a presentation of $S$ as the blow-up of six general points, giving a ``six'' in the $27$ lines of the cubic surface.
Clearly, there are also $72$ sixes.
The twisted cubic $\Gm'$ corresponds to the double six of $\Gm$ (see Section 9 of \cite{dolgachev}).
\end{rem}

\section{$Z$ as a moduli space of stable aCM bundles}\label{sec:stableaCM}

In this section we prove Part~\eqref{item:main3} of the Main Theorem.
In particular, $Y$ will always be a cubic fourfold not containing a plane and $\Gm$ will be an aCM curve in $M_3(Y)$.
The key point consists in showing that $M_\Gm$ is Gieseker stable (Section~\ref{subsec:stabMGamma}).
The theorem is finally proved in Section~\ref{subsec:proofthmmain1}.

\subsection{Stability of $M_\Gm$}\label{subsec:stabMGamma}

In order to study the stability of the vector bundle $M_\Gm$ in Proposition~\ref{prop:stabaCM} we first study the stability of the sheaf $K_\Gm$  (see \eqref{eqn:laM}).

\begin{lem}\label{lem:laKirred}
Let $\Gm$ be an aCM twisted cubic in $Y$ and suppose that $S=\gen{\Gm}\cap Y$ is an integral surface.
The torsion-free sheaf $K_\Gm$ defined in Lemma~\ref{lem:laKirred} is a Gieseker stable sheaf with reduced Hilbert polynomial
\begin{equation}\label{eqn:polinomiKC}
	p(K_\Gm,n):=\frac1{8}n^4+\frac{3}{4}n^3+\frac{7}{8}n^2+\frac{1}{4}n.
\end{equation}
\end{lem}

\begin{proof}
For sake of simplicity, let us just write $K$ for $K_\Gm$.
We need to show that the reduced Hilbert polynomial of any non-trivial proper saturated subsheaf $A\subset K$ satisfies $p(A,n)< p(K,n)$.
As $K$ has rank three, the subsheaf $A$ has rank one or rank two.

\smallskip

\noindent{\it Case $\rk(A)=1$.}
By letting $G_\Gm(H)$ play the role of $\I_{\Gm/S}(2H)$, we can apply exactly the same arguments as in the proof of case $\rk(A)=1$ of Proposition~\ref{prop:stabF} in the present situation (see \eqref{eqn:laGC} and Lemma~\ref{lem:laKirred}).
Hence, we get that the only possible destabilizing ideal sheaf is $\I_{W/Y}$ where $W$ contains $S$ and
\[
P(\I_{W/Y},n)\leq \frac 18 n^4 + \frac 34 n^3 + \frac 38 n^2 + \frac 34 n< p(K,n)
\]
which completes the analysis in this case.

\smallskip

\noindent{\it Case $\rk(A)=2$.}
Again, the same arguments as in the proof of case $\rk(A)=2$ of Proposition~\ref{prop:stabF} show that we have a subscheme $W\subset Y$ of codimension two such that and
the following commutative diagram
\[
\xymatrix{
&0&0&0&\\
0\ar[r]&\I_{W/Y}\ar[r]\ar[u]&\O_Y\ar[r]\ar[u]&\O_W\ar[r]\ar[u]&0\\
0\ar[r]&K\ar[r]\ar[u]&\O_Y^{\oplus 3}\ar[r]\ar[u]&G_{\Gm}(H)\ar[r]\ar[u]&0\\
0\ar[r]&A\ar[r]\ar[u]&\O_Y^{\oplus 2}\ar[r]\ar[u]&Q'\ar[r]\ar[u]&0\\
&0\ar[u]&0\ar[u]&0\ar[u]&\\
}
\]
with exact lines and columns.
If $\I_{W/Y}$ is assumed to be a destabilizing quotient of $K$, we must have $p(\I_{W/Y},n)\leq p(K,n)$ or, equivalently, $P(\O_W,n)\geq p(\O_Y,n)-p(K,n)=n^2+2n+1$.
This shows that $\ch_2(W)\geq \ch_2(G_\Gm(H))$.
In this case, the support of the kernel $Q'$ is $1$-dimensional, which is impossible since $G_{\Gm}(H)$ is pure of dimension two.
Hence, we are done with the second case as well.
\end{proof}

Having this, we can finally go back to the vector bundles $M_\Gm$.
In particular, we get the following.

\begin{prop}\label{prop:stabaCM}
Let $\Gm$ be an aCM twisted cubic in $Y$ and suppose that $S=\gen{\Gm}\cap Y$ is an integral surface.
Then the rank six aCM bundle $M_\Gm$ is Gieseker stable with
$\ch(M_{\Gm})=\mathbf{v}_3$.
\end{prop}

\begin{proof}
The Chern character of $M_C$ was computed in \eqref{eqn:chM}.
We are left to prove stability.
For sake of simplicity, let us just write $M$ for $M_\Gm$.
Since $M$ is a vector bundle, we need to show that the reduced Hilbert polynomial of any non-trivial proper saturated reflexive subsheaf $A\subset M$ satisfies $p(A,n)< p(M,n)$.
If it is not the case, we have
\begin{equation*}
0\longrightarrow A\longrightarrow M\longrightarrow B\longrightarrow 0
\end{equation*}
with $B$ a torsion-free sheaf and $p(A,n)\geq p(M,n)$.
Applying the functor $\cHom(-,\O_Y(-H))$ to the exact sequence we get also
\[
0\longrightarrow \cHom(B,\O_Y(-H))\stackrel{\psi}\longrightarrow M_{\Gm'}\longrightarrow \cHom(A,\O_Y(-H)),
\]
where $\cHom(B,\O_Y(-H))$ is reflexive and $H^{4-k}.\ch_k(\coker \psi)\leq H^{4-k}.\ch_k(\cHom(A,\O_Y(-H)))$ for $k= 2$ and they are equal for $k\leq 1$.
Hence, if $A$ is a reflexive destabilizing sheaf of $M_{\Gm}$ of rank four or five, then there is subsheaf of $M_{\Gm'}$
with the same $\ch_1$ and possibly bigger $\ch_2$ which is a destabilizing subsheaf respectively of rank two and one.

This shows that we just need to analyze the cases $\rk(A)=1,2,3$ and show that, in order to exclude them, it is enough work with the Chern character truncated at degree smaller or equal to two.

\bigskip

\noindent{\em Case 1: $\rk(A)=1$.}
As $A$ is a line bundle, it has the form $A=\O_{Y}(m)$ for some twist $m\in \ZZ$.
The leading terms of $p(A,n)$ are $\frac{1}8n^4+(\frac m2+\frac 34)n^3+\ldots$.
Since $A$ is a subsheaf of $\O_Y^{\oplus 9}$ as well, one has $m\leq 0$.
But if $m<0$, then $A$ is not destabilizing.
Hence, only the case $m=0$ and $A=\O_{Y}$ requires further consideration.
But in this case, since $h^0(Y,M)=0$, we get a contradiction.

\smallskip

When $\rk(A)\geq 2$, we have the following commutative diagram
\begin{equation}\label{eqn:stab}
 \xymatrix@!C=2.5pc{
&0&0&0&\\
0\ar[r]&B_1\ar[r]\ar[u]&B\ar[r]\ar[u]&B_2\ar[r]\ar[u]&0\\
0\ar[r]&\O_Y(-H)^{\oplus 3}\ar[r]\ar[u]&M\ar[r]\ar[u]&K_{\Gm}\ar[r]\ar[u]&0\\
0\ar[r]&A_1\ar[r]\ar[u]&A\ar[r]\ar[u]&A_2\ar[r]\ar[u]&0\\
&0\ar[u]&0\ar[u]&0\ar[u]&\\
}
\end{equation}
with exact lines and columns and where $A_i$ and $B_i$ are possibly zero.
Since we have assumed that $B$ is torsion-free, then $B_1$ is also torsion-free.
Then, by the first vertical exact sequence, we have that,
$A_1$ is reflexive.

\bigskip

\noindent Now we study the remaining cases.
We use the following convention: the {\em Case $i.j$}, will refer to the case where $\rk(A)= i$ and $\rk(A_2)=j$.

\medskip

\noindent{\em Case 2.0: $\rk(A_2)=0$.}
Since $K_\Gm$ is torsion free, $A_2=0$.
Hence, $A$ is a subsheaf of the semistable sheaf $\O_Y(-H)^{\oplus 3}$ and $p(A,n)\leq p(\O_Y(-H),n)<p(M,n)$, so $A$ does not destabilize $M$.

\medskip

\noindent{\em Case 2.1: $\rk(A_1)=\rk(A_2)=1$.}
In that situation, in order to destabilize we need that $c_1(A_1)+c_1(A_2)\geq-1$.
Since $K_\Gm$ is Gieseker stable by Lemma~\ref{lem:laKirred}, the only possibility is that $A_1\cong \O_Y(-H)$ and $A_2=\I_{W/Y}$ for
some subscheme $W\subset Y$ of codimension greater or equal than two.
Moreover as $h^0(Y,K_\Gm)=0$, we have that $W\subset Y$ is non-empty.
By the presentation \eqref{eqn:laK3I}, and since the cubic surface $S$ is irreducible and reduced, this forces $W$ to contain $S$, so that
$P(\O_W,n)\geq P(\O_S,n)=\frac12(3n^2+3n+2)$, and $P(\I_{W/Y},n)\leq \frac 18 n^4 + \frac 34 n^3 + \frac 38 n^2 + \frac 34 n$.
Thus
\[
P(A,n)\leq \frac 18 n^4 + \frac 12 n^3 + \frac 38 n^2 + \frac 34 n <p(M,n)
\]
which completes the analysis in this case.

\medskip

\noindent{\em Case 2.2: $\rk(A_2)=2$.}
In that case, $A\cong A_2$.
In order to destabilize we need $c_1(A)\geq -1$.
If $c_1(A)\geq 0$, then $c_1(A)=0$ and $A$ is necessarily the extension of two ideals of subschemes of codimension at least two (by the stability of $K_\Gm$ and \eqref{eqn:laKGm}).
Moreover, since $A$ is reflexive, we have
\[
	0\to \O_Y\to A\to \I_{W/Y}\to 0
\]
where $W$ has codimension greater or equal than two.
Hence, $h^0(Y,A)>0$ contradicting $h^0(Y,M)=0$.

Now we need to consider the case $c_1(A)=-1$.
As before, since $A$ is reflexive and semistable, we have
\[
	0\to \O_Y(-H)\to A\to \I_{W/Y}\to 0
\]
with $W$ of codimension two to let the extension be non-trivial.
Since $Y$ does not contain a plane (nor a quadric), $\ch_2(\O_W)\cdot H^2\geq \ch_2(\O_S)\cdot H^2$.
Thus, $P(\O_W,n)\geq P(\O_S,n)=\frac12(3n^2+3n+2)$, and as before, $P(\I_{W/Y},n)\leq \frac 18 n^4 + \frac 34 n^3 + \frac 38 n^2 + \frac 34 n$.
Then
\[
P(A,n)\leq \frac 18 n^4 + \frac 12 n^3 + \frac 38 n^2 + \frac 34 n <p(M,n)
\]
which completes the analysis in this case.

\bigskip

\noindent{\em Case 3.0: $\rk(A_2)=0$.}
Since $K_\Gm$ is torsion free, $A_2=0$.
Hence, $A$ is a subsheaf of the semistable sheaf $\O_Y(-H)^{\oplus 3}$ and $p(A,n)\leq p(\O_Y(-H),n)<p(M,n)$, so $A$ does not destabilize $M$.

\medskip

\noindent{\em Case 3.1: $\rk(A_2)=1$.}
Since $\O_Y(-H)^{\oplus 3}$ is semistable, $p(A_1,n)\leq p(\O_Y(-H),n)$.
On the other hand, also $K_\Gm$ is semistable, so $p(A_2,n)\leq p(K_\Gm,n)$.
Then
\[
p(A,n)\leq \frac 23 p(\O_Y(-H),n) + \frac 13 p(K_\Gm,n)= \frac 18 n^4 + \frac{5}{12}n^3+\ldots < p(M,n).
\]
Therefore, $A$ does not destabilize $M$.

\medskip

\noindent{\em Case 3.2: $\rk(A_2)=2$.}
If $A$ destabilizes $M$, in particular $-\tfrac{1}{3}\leq \mu(A)$.
Since $A$ and $K_\Gm$ are semistable, $\mu(A)\leq\mu(A_2)\leq \mu(K_\Gm)=0$.
As $2\mu(A_2)$ is an integer, $\mu(A_2)=c_1(A_2)=0$.
Since $A_1$ is a reflexive sheaf of rank one, $A_1$ is a line bundle, more precisely, the only possibility is $A_1\cong \O_Y(-H)$.

Note that $A_2$ in \eqref{eqn:stab} is the extension of two ideals
\[
 0\to \I_{W'/Y}\to A_2\to \I_{W''/Y}\to 0,
\]
with $S\subseteq W',W''$, and $\codim W'$, $\codim W''$ at least two (by \eqref{eqn:laK3I}).
Then, by \eqref{eqn:laKGm}, $B_2$ in \eqref{eqn:stab} has either torsion in codimension at least two or it is isomorphic to $\I_{\Gm/Y}$.
Note that
\[
\Ext^1(\I_{\Gm/Y},\O_Y(-H))\cong H^3(Y,\I_{\Gm/Y}(-2H))=0.
\]
Hence, since $\O_Y(-H)$ is not a direct summand of $M$, the only possibility is that $B_2$ has torsion in codimension $k\geq 2$.
But then, $\cExt^{k-1}(A,\O_Y)\cong \cExt^k(B,\O_Y)\cong \cExt^k(B_2,\O_Y)$ is supported in codimension $k$.
So we get a contradiction with $A$ being reflexive.

\medskip

\noindent{\em Case 3.3: $\rk(A_2)=3$.}
In that case $A\cong A_2$ is semistable and reflexive.
If $c_1(A_2)=0$, then either $A_2=K_\Gm$, which is impossible since \eqref{eqn:laM} does not split, or $B_2$ has torsion in codimension at least two contradicting $A$ being reflexive as in the previous case.

Hence, we can suppose $c_1(A)=-1$.
Suppose that $B_2$ is a pure sheaf supported on $H$.
Hence, $B_2=\I_{W/H}(aH)$ for some $a\in \ZZ$ and $W\subset H$ of dimension at most one.
Then $\cExt^1(B_2,\O_Y)\cong \O_H(-(a+1)H)$.
A simple computation using \eqref{eqn:laKGm} shows that $\cExt^1(K_\Gm,\O_Y)$ has rank two on $S$.
Thus, the natural exact sequence
\[
 \O_H(-(a+1)H)\cong \cExt^1(B_2,\O_Y)\to \cExt^1(K_\Gm,\O_Y)\to \cExt^1(A_2,\O_Y),
\]
implies that $S\subseteq \cExt^1(A_2,\O_Y)$ contradicting the reflexiveness of $A_2$.
\end{proof}

Let $\mathfrak{N}$ be the moduli space of Gieseker semistable sheaves over $Y$ with reduced Hilbert polynomial \eqref{eqn:HilbertPolyMC} and denote by $\mathfrak{N}^{\rm aCM}\subseteq\mathfrak{N}$ the open subset of Gieseker stable aCM bundles.
The result above immediately implies the following.

\begin{cor}\label{cor:nonempty}
The open subset $\mathfrak{N}^{\rm aCM}$ is non-empty.
\end{cor}

In particular, we get a new family of Gieseker stable aCM bundles on any cubic fourfold not containing a plane.
As a consequence of the discussion in the next section, such a family (or rather its smooth locus) has dimension eight.

\subsection{Proof of Part~\eqref{item:main3} of the Main  Theorem}\label{subsec:proofthmmain1}

Consider the open subset $U$ in $\mM_1$ that corresponds to $Z'(Y)\setminus D$.
In other words, a point in $U$ is a sheaf $F_\Gm$, where $\Gm$ is an aCM curve in $Y$.
By applying the functor
$\Upsilon=\cat{L}_{\O_Y}\left(-\otimes \O_Y(H)\right)[-1]$ to $F_\Gm$ and by Proposition~\ref{prop:stabaCM}, we get a morphism $U\to\mathfrak{N}^{\rm aCM}$.
Note that this is the same as applying $\cat{L}_{\O_Y}\left(\cat{L}_{\O_Y}\left(-\right)\otimes \O_Y(H)\right)[-2]$ to $\I_{\Gm/S}(2H)$.

To show that this provides an isomorphism between $U$ and an open subset of a component of the closure of $\mathfrak{N}^{\rm aCM}$ in $\mathfrak{N}$, we just observe that $\Upsilon$ is an autoequivalence of $\cat{T}_Y$ (see Remark~\ref{rmk:autoeq}) and thus, given two aCM generalized twisted cubics $\Gm_1$ and $\Gm_2$ on $Y$, we have $F_{\Gm_1}\cong F_{\Gm_2}$ if and only if $M_{\Gm_1}\cong M_{\Gm_2}$.
For the same reason, $\Upsilon$ induces an isomorphism between $\Ext^1(F_\Gm,F_\Gm)$ and $\Ext^1(M_\Gm,M_\Gm)$ for any aCM generalized twisted cubic $\Gm$ on $Y$.
This concludes the proof of the theorem.

\section{$Z$ is generically a moduli space of tilt-stable objects}\label{sec:tiltstable}

In this section we prove Part~\eqref{item:main2b} of the Main Theorem.
The proof is based on a wall-crossing argument with respect to tilt-stability.
This is preceded by an introductory discussion concerning general facts about tilt-stability.

\subsection{Tilt-stability}\label{subsec:tilt}

Let $Y$ be a smooth cubic fourfold and let $\alpha,\beta\in\RR$, with $\alpha>0$, be two parameters.
We will sometimes refer to the half-plane
\[
\left\{(\alpha,\beta)\in\RR^2\,:\, \alpha>0 \right\}
\]
as the $(\alpha,\beta)$-plane.

For a coherent sheaf $E$ on $Y$, the vector
\[
\ch^{\beta}(E) = e^{-\beta H} \ch(E)\in H^*(Y,\mathbb{R})
\]
is called \emph{twisted Chern character} of $E$.
Unraveling the definition, it is easy to compute the degree zero, one and two parts of $\ch^\beta(E)$, which will be extensively used later:
\[
\ch^\beta_0=\ch_0=\rk(E) \qquad \ch^\beta_1=\ch_1-\beta H \ch_0\qquad \ch^\beta_2=\ch_2- \beta H\ch_1+\frac{\beta^2 H^2}{2} \ch_0.
\]

For a coherent sheaf $E$ on $Y$, we set
\[
\mu_{\beta}(E) = \begin{cases}
 + \infty, & \text{ if }\ch^\beta_0(E)=0,\\
\frac{H^3\ch^\beta_1(E)}{H^4 \ch^\beta_0(E)}, & \text{ otherwise.}
 \end{cases}
\]
This slope function is a minor variation of the usual slope function.
The induced notion of stability is exactly the same as slope-stability for torsion-free sheaves.
Torsion sheaves instead are all semistable.

Following \cite{BK3}, there exists a \emph{torsion pair} $(\mathbf{T}^{\beta}, \mathbf{F}^{\beta})$ in $\Coh(Y)$ where $\mathbf{T}^{\beta}$ and $\mathbf{F}^{\beta}$ are the extension-closed subcategories of $\Coh(Y)$ generated by $\mu_{\beta}$-stable sheaves of positive and non-positive slope, respectively.
Let $\Coh^{\beta}(Y) \subseteq \Db(Y)$ denote the extension-closure
\[
\Coh^{\beta}(Y) := \langle \mathbf{T}^{\beta},\mathbf{F}^{\beta}[1] \rangle.
\]
It turns out that $\Coh^{\beta}(Y)$ is an abelian category which is the heart of a bounded t-structure on $\Db(Y)$.
With respect to the t-structure given by the abelian category of coherent sheaves, an object $E$ in $\Coh^\beta(Y)$ has at most two non-zero cohomology sheaves, namely $\H^{-1}(E)\in\mathbf{F}^{\beta}$ and $\H^0(E)\in\mathbf{T}^{\beta}$.

Following \cite{BMT}, for an object $E$ in $\Coh^{\beta}(Y)$, we define
\begin{equation*}
\nu_{\alpha,\beta}(E) = \begin{cases}
 + \infty, & \text{ if }H^{3} \ch^{\beta}_1(E) = 0,\\
\frac{H^2 \ch^\beta_2(E) - \frac{\alpha^2}{2} H^4 \ch^\beta_0(E)}{H^3\ch^\beta_1(E)}, & \text{ otherwise.}
\end{cases}
\end{equation*}
An object $E \in \Coh^{\beta}(Y)$ is \emph{tilt-(semi)stable} (or \emph{$\nu_{\alpha,\beta}$-(semi)stable}) if, for all non-trivial subobjects
$F$ of $E$,
\[
\nu_{\alpha,\beta}(F) < (\le) \nu_{\alpha,\beta}(E/F).
\]
Tilt-stability is a good very\footnote{In more recent papers, like \cite{BLMS}, a very weak stability condition is called a weak stability condition.} weak stability condition satisfying a Bogomolov--Gieseker inequality, in the sense of \cite[Appendix~B]{BMS} and \cite[Definition~2.13]{PT}.
In particular, it satisfies the weak see-saw property, Harder--Narasimhan filtrations exist \cite[Proposition~B.2]{BMS}, and the dependence on the two parameters $\alpha$ and $\beta$ is continuous with a locally-finite wall and chamber structure \cite[Propositions~B.2 \& B.5]{BMS}.
Note that the results in \emph{loc.~cit.} are only stated for threefolds, but the argument is identical in the fourfold case, actually in any dimension.

The next two results are also only stated for threefolds (or surfaces), but the proofs remain valid in any dimension.

\begin{thm}[{\cite[Theorem~7.3.1]{BMT} and \cite[Theorem~3.5]{BMS}}] \label{thm:BGwithoutch3}
Let $E\in\Coh^\beta(Y)$ be a $\nu_{\alpha,\beta}$-semistable object.
Then
\[
\Delta(E) := (H^3\ch_1(E))^2 - 2 (H^4 \ch_0(E)) (H^2\ch_2(E)) \geq 0.
\]
In addition, if $\H^{-1}(E)\neq0$, then $\Delta(E)=0$ if and only if $E\cong \O_Y(dH)^{\oplus r}[1]$ for some $d\in\ZZ$ and $r\in\ZZ_{>0}$.
\end{thm}

\begin{proof}
We only prove the second statement.
If $E$ is $\nu_{\alpha,\beta}$-semistable with $\H^{-1}(E)\neq0$ and $\Delta(E)=0$, then by \cite[Corollary~3.11, (c)]{BMS} and the proof of \cite[Proposition~3.12]{BMS} (which, in turn, uses \cite[Theorem~2]{Si}),
$\H^0(E)$ must be zero, and $E \cong \H^{-1}(E)[1]$, where $\H^{-1}(E)$ is a slope-semistable vector bundle with $\Delta(\H^{-1}(E))=0$.
Thus, \cite{Lu} applies (see also \cite[Theorem~4.7]{Ko}), and $\H^{-1}(E)$ is a projectively flat vector bundle.
Since $Y$ is simply-connected, again by \cite[Theorem~2]{Si}, $\H^{-1}(E)$ must be a direct sum of line bundles, as in the statement.
The fact that line bundles are stable is nothing but \cite[Proposition~7.4.1]{BMT}; this shows the converse implication.
\end{proof}

For an object $E$ in $\Db(Y)$, we denote by $\ch_{\leq 2}(E)$ the truncation at degree two of the Chern character $\ch(E)$.

\begin{thm}[{\cite[Proposition~14.2]{BK3}}] \label{thm:ch2}
Let $E\in\Db(Y)$ be such that the vector $\ch_{\leq 2}(E)$ is primitive, $\ch_0(E)>0$, and $H^3\ch_1^\beta(E)>0$.
Then $E$ is $\nu_{\alpha,\beta}$-stable for $\alpha\gg0$ if and only if $E$ is Gieseker stable.
\end{thm}

Combining Theorem~\ref{thm:ch2} with Part~\eqref{item:main2a} of the Main Theorem, we get the following.

\begin{cor}\label{cor:tilt}
Let $Y$ be a smooth cubic fourfold not containing a plane.
Then any $F_\Gm$ in $Z'(Y)$ is $\nu_{\alpha,\beta}$-stable for $\alpha\gg 0$ and $\beta<0$.
\end{cor}

We will actually need a stronger and more precise version of the previous result. When $Y$ is a very general cubic fourfolds, we need to show that $\alpha\gg0$ can be chosen uniformly for all generalized twisted cubics $\Gm$ (see Corollary~\ref{cor:precisemoduli} below).

\subsection{The first wall}\label{subsec:thewall}

In this section we compute exactly the first wall in Corollary~\ref{cor:tilt} above.
Namely, we fix the primitive vector $\mathbf{v}_2'=\left(3,0,-H^2\right)$.
We know that for $\alpha\gg0$ and $\beta<0$, tilt-stable objects correspond to Gieseker stable sheaves $E$ with $\ch_{\leq2}(E)=\mathbf{v}_2'$.
We want to determine, in the region $\alpha>0$ and $\beta<0$, the first locus where tilt-stability changes.
We will use extensively results from \cite{S}.
As before, these results were proved in the threefold case only; the results still hold with identical proofs in our fourfold case (and again, in any dimension as well).

From now on, we assume that $Y$ is \emph{very general}
\footnote{While the walls in the $(\alpha,\beta)$-plane make sense for any cubic fourfold, we assume that $Y$ is very general since it may happen that the wall we find is not the first wall for special cubic fourfolds.
To still recover it as the first wall, we would probably need to consider a generalized tilt-stability function in which $H^2\ch_2^\beta$ is deformed to $\gamma\ch_2^\beta$, where $\gamma$ is allowed to vary in $H^4(Y,\mathbb{Z})\cap H^{2,2}(Y)$.}, namely that the algebraic part $H^4(Y,\mathbb{Z})\cap H^{2,2}(Y)$ of the cohomology group $H^4(Y,\mathbb{Z})$ is generated by the class of a smooth cubic surface $H^2$.
We start by recalling the definition of a wall.
We consider the rank three lattice
\[
\Gamma := \ZZ \gen{\ch_{\leq 2}(E) \,:\, E\in\Db(Y)}  \subset \ZZ \gen{(1,0,0), (0,H,0), (0,0,\tfrac{1}{2} H^2)}
\]
generated by the truncated Chern characters of objects in $\Db(Y)$.

\begin{defn}
Let $\mathbf{w}\in\Gamma$.
A \emph{numerical wall} for $\mathbf{w}$ in the $(\alpha,\beta)$-plane is a proper non-trivial solution set of an equation
\[
\nu_{\alpha,\beta}(\mathbf{w}) = \nu_{\alpha,\beta}(\mathbf{u}),
\]
for a vector $\mathbf{u}\in\Gamma$.
Given an object $E\in \Coh^\beta(Y)$, a numerical wall for $\mathbf{w}=\ch_{\leq 2}(E)$ is called an \emph{actual wall for $E$}, if for each point $(\alpha,\beta)$ of the numerical wall there is an exact sequence of semistable objects
\[
0\to F \to E \to G \to 0
\]
in $\Coh^\beta(Y)$, where $\nu_{\alpha,\beta}(F)=\nu_{\alpha,\beta}(G)$ numerically defines the wall.
\end{defn}
Since $\nu_{\alpha,\beta}$ only depends on $\Gamma$, the previous definition makes sense.

In the $(\alpha,\beta)$-plane, numerical walls for a vector $\mathbf{w}=(r,cH,sH^2)$ are very easy to describe (see, for example, \cite{M} and \cite[Theorem~3.3]{S}).
That is, there is a unique straight wall corresponding to $\beta=\frac cr$.
Moreover, all other walls are strictly nested semicircles, whose centers converge to the two points $\beta=\overline{\beta}_1,\overline{\beta}_2$ for which $\ch_2^{\beta}(\mathbf{w})=0$.
Thus, it makes sense to speak about the region inside or outside a numerical wall (see Figure~\ref{fig:firstwall}).

In our case, $\mathbf{w}=\mathbf{v}'_2$.
The unique straight wall is at $\beta=0$.
We are interested in the region $\beta<0$, so that the objects $F_\Gm$ are in $\Coh^{\beta}(Y)$ and we can apply Corollary~\ref{cor:tilt}.
In this region, the point of accumulation for the center of all walls is $\overline{\beta} = -\frac{\sqrt{6}}{3}$.
By Corollary~\ref{cor:tilt}, we know there is a largest semicircular wall for any sheaf $F_{\Gm}$; the next lemma provides a uniform bound for this wall, which will be an actual wall for sheaves $F_\Gm$ corresponding to non-CM curves.

\begin{lem}\label{lem:firstwall}
Let $\mathbf{v}'_2=(3,0,-H)\in\Gamma$.
The largest numerical wall $W_0$ for $\mathbf{v}'_2$ in the region $\beta<0$ is given by the equation
\[
\nu_{\alpha,\beta}(\O_Y(-H)) = \nu_{\alpha,\beta}(\mathbf{v}'_2).
\]
\end{lem}

\begin{proof}
By using Theorem~\ref{thm:BGwithoutch3}, the following inequalities have to be satisfied for a wall $W$ to exist (see \cite[Section 5.3]{S}).
Let $(\alpha,\beta)\in W$ and let $0\to F \to E \to G \to 0$ be an exact sequence defining a wall with $\ch_{\leq 2}(E)=\mathbf{v}'_2$.
Then
\begin{align*}
& \nu_{\alpha,\beta}(F)=\nu_{\alpha,\beta}(E)=\nu_{\alpha,\beta}(G)\\
& 0 < H^3\ch_1^\beta(F),\, H^3\ch_1^\beta(G) < H^3\ch_1^\beta(E)\\
& 0\leq \Delta(F),\, \Delta(G) < \Delta(E).
\end{align*}
Moreover, we need to impose $\ch_{\leq 2}(F),\ch_{\leq 2}(G)\in\Gamma$.

All of these put enough inequalities to implement a computer program to compute the possible classes of $F$ and $G$ in $\Gm$.
For example, the code for a concrete implementation in SAGE \cite{SAGE} can be found at Benjamin Schmidt's webpage:
\[
\verb+https://sites.google.com/site/benjaminschmidtmath/research+
\]

By doing the computation at $\beta=-1$, this shows that there is no numerical wall intersecting the line $\beta=-1$, that is, no classes satisfy the inequalities.
The wall $W_0$ is given by the equation
\[
\alpha^2 + \left(\beta + \frac 56 \right)^2 = \frac{1}{36}.
\]
It passes through the point $(0,-1)$.
Since walls are strictly nested, this is the largest wall
(see Figure~\ref{fig:firstwall}).
\end{proof}

Hence, for a very general cubic fourfold, we have the following precise version of Corollary~\ref{cor:tilt}.

\begin{cor}\label{cor:precisemoduli}
Let $\mathbf{v}'_2 = (3,0,-H^2)$ and let $(\alpha,\beta)$ be outside the numerical wall $W_0$ and such that $\beta<0$.
Then for an object $E\in\Db(Y)$ with $\ch_{\leq 2}(E)=\mathbf{v}'_2$, the following are equivalent:
\begin{enumerate}
\item $E$ is $\nu_{\alpha,\beta}$-semistable;
\item $E$ is $\nu_{\alpha,\beta}$-stable;
\item $E$ is a torsion-free Gieseker-stable sheaf.
\end{enumerate}
\end{cor}

A similar argument as in Lemma~\ref{lem:firstwall} shows which classes in $\Gamma$ give the numerical wall $W_0$.

\begin{lem}\label{lem:classeswall}
Let $\mathbf{u}\in\Gamma$ be such that $\nu_{\alpha,\beta}(\mathbf{u}) = \nu_{\alpha,\beta}(\mathbf{v}'_2)$ for all $(\alpha,\beta)\in W_0$.
Then $\mathbf{u}$ is one of the following classes:
\begin{align*}
& (-1, H, -H^2/2),\, (-2, 2H, -H^2),\, (-3, 3H, -3H^2/2),\, (-4, 4H, -2H^2),\\
& (-5, 5H, -5H^2/2),\, (-6, 6H, -3H^2),\, (4, -H, -H^2/2),\,  (5, -2H, 0), \\
&  (6, -3H, H^2/2),\,  (7, -4H, H^2),\, (8, -5H, 3H^2/2),\, (9, -6H, 2H^2).
\end{align*}
\end{lem}

\begin{proof}
This is a straightforward computation, which again can be implemented in a computer program.
We only need to check which classes in $\Gamma$ satisfy the inequalities in Lemma~\ref{lem:firstwall} for $(\alpha,\beta)=(\frac16, -\frac56)$. 
Here we use again the SAGE code at Benjamin Schmidt's webpage.
\end{proof}

This gives us how the sheaves $F_\Gm$ destabilize at the wall $W_0$, determining for which sheaves $F_\Gm$, $W_0$ is an actual wall.

\begin{prop}\label{prop:tiltstabilityFC}
Let $(\alpha,\beta)\in W_0$.
\begin{enumerate}
\item Let $C$ be an aCM twisted cubic curve in $Y$.
Then $F_\Gm$ is $\nu_{\alpha,\beta}$-stable.
\item Let $C$ be a non-CM cubic curve in $Y$.
Then $F_\Gm$ is $\nu_{\alpha,\beta}$-semistable and a Jordan--H\"older filtration in $\Coh^\beta(Y)$ is given by
\begin{equation}\label{eq:destabilizing}
0\to N_\Gm \to F_\Gm \to \O_Y(-H)[1] \to 0.
\end{equation}
\end{enumerate}
\end{prop}

\begin{proof}
By Corollary~\ref{cor:precisemoduli}, all $F_\Gm$ are semistable at the numerical wall $W_0$.
Assume that $F_\Gm$ is not stable.
Then, by Lemma~\ref{lem:classeswall}, there is a stable quotient $F_\Gm \rra Q$ with $\ch_{\leq 2}(Q)=a \cdot (-1,H,-H^2/2)$ for $a=1,\ldots,6$.
But, by Theorem~\ref{thm:BGwithoutch3}, since $\ch_0(Q)<0$ (and so $\H^{-1}(Q)\neq0$) and $\Delta(Q)=0$, we have $Q=\O_Y(-H)[1]$ (and so $a=1$).

We can now use the computation in Remark~\ref{rmk:nonT}.
If $C$ is aCM, then $\Hom(F_\Gm,\O_Y(-H)[1])=0$, and so $F_\Gm$ is stable.
If $C$ is non-CM, then $\Hom(F_\Gm,\O_Y(-H)[1])\cong \CC$.
Hence, the kernel $N_\Gm:=\ker(F_\Gm\rra \O_Y(-H)[1])$ satisfies 
$\Hom(N_\Gm,\O_Y(-H)[1])=0$; therefore, by the same argument, $N_\Gm$ is  $\nu_{\alpha,\beta}$-stable and \eqref{eq:destabilizing} gives a Jordan--H\"older filtration for $F_\Gm$.
\end{proof}

Finally, by choosing $(\alpha,\beta)$ inside the wall $W_0$ and sufficiently close to the wall itself (actually, again by a straightforward computation, it is enough to pick $\beta=-5/6$ and arbitrary $0<\alpha<1/6$; we do not need this), we have the following stable objects.

Let us denote by $F_\Gm'=\mathrm{pr}(F_\Gm)\in\cat{T}_Y$, where $\mathop{\mathrm{pr}}\colon\Db(Y)\to \cat{T}_Y$ is the projection functor (left adjoint to the inclusion).
Note that $F_\Gm'=\cat{R}_{\O_Y(-H)}(F_\Gm)$; in particular, if $\Gm$ is an aCM twisted cubic curve, then $F_\Gm'=F_\Gm$.

\begin{prop}\label{prop:tiltstabilityFCafterwall}
Let $\beta=-5/6$ and $\alpha = 1/6-\epsilon$ for $\epsilon>0$ sufficiently small.
\begin{enumerate}
\item Let $C$ be an aCM twisted cubic curve in $Y$.
Then $F_\Gm'=F_\Gm$ is $\nu_{\alpha,\beta}$-stable.
\item Let $C$ be a non-CM cubic curve in $Y$.
Then $F_\Gm'$ is $\nu_{\alpha,\beta}$-stable and fits into a non-split exact sequence in $\Coh^\beta(Y)$
\[
0\to \O_Y(-H)[1] \to F_\Gm' \to N_\Gm  \to 0.
\]
\end{enumerate}
\end{prop}

\begin{proof}
The first part follows immediately by openness of tilt-stability
\cite[Proposition~B.5]{BMS}.
For the second, the fact that $F_C'$ fits into a non-split short exact sequence as in the statement is a direct computation by using Remark~\ref{rmk:nonT}.
To prove it is $\nu_{\alpha,\beta}$-stable, we first observe that $F_\Gm'$ is semistable at the wall $W_0$.
Then, as in the proof of Proposition~\ref{prop:tiltstabilityFC}, if $F'_\Gm$ is not stable inside the circle, it must have $\O_Y(-H)[1]$ as quotient, which is impossible.
\end{proof}

\begin{figure}[!h]\definecolor{xdxdff}{rgb}{0.49,0.49,1}
\definecolor{zzqqqq}{rgb}{1,0,0.2}
\definecolor{zzzzqq}{rgb}{0.6,0.6,0}
\definecolor{qqqqcc}{rgb}{0,0,0.8}
\definecolor{zzqqqq}{rgb}{0.6,0,0}
\definecolor{ffffff}{rgb}{1,1,1}
\definecolor{uququq}{rgb}{0.25,0.25,0.25}
\definecolor{qqqqff}{rgb}{0,0,1}
\begin{tikzpicture}[line cap=round,line join=round,>=triangle 45,x=4.0cm,y=4.0cm]
\clip(-0.07,-1.45) rectangle (1.64,0.05);
\draw (0,0)-- (1.2,0);
\draw [->] (0,0) -- (1.2,0);
\draw [->] (0,0) -- (0,-1.4);
\draw [shift={(0,-0.83)},color=zzqqqq]  plot[domain=-1.57:1.57,variable=\t]({1*0.17*cos(\t r)+0*0.17*sin(\t r)},{0*0.17*cos(\t r)+1*0.17*sin(\t r)});
\fill[color=zzzzqq,fill=zzzzqq,fill opacity=0.03] (0,0) -- (1.90,0) -- (1.90,-1.77) -- (0,-1.77) -- cycle;
\draw [shift={(0.01,-0.83)},color=ffffff,fill=ffffff,fill opacity=1.0]  (0,0) --  plot[domain=-1.57:1.57,variable=\t]({1*0.155*cos(\t r)+0*0.155*sin(\t r)},{0*0.155*cos(\t r)+1*0.155*sin(\t r)}) -- cycle ;
\begin{scriptsize}
\draw [domain=0.12792578123625104:1.9415421334702613] plot(\x,{(-0.23-0*\x)/0.28});
\draw (0.13,-0.15) node[anchor=north west] {Stability in {\color{zzzzqq} $A$} described in Corollary~\ref{cor:precisemoduli}.};
\draw (0.13,-0.25) node[anchor=north west] {Stability in {\color{zzqqqq} $B$} described in Proposition~\ref{prop:tiltstabilityFC}.};
\draw (0.13,-0.35) node[anchor=north west] {Stability in {\color{qqqqcc} $C$} described in Proposition~\ref{prop:tiltstabilityFCafterwall}.};
\end{scriptsize}
\begin{tiny}
\draw (1.17,0.00) node[anchor=south west] {$\alpha$};
\draw (0.02,-1.33) node[anchor=north west] {$\beta<0$};
\fill [color=qqqqff] (0,-0.83) circle (1.pt);
\fill [color=qqqqff] (0,-0.81) circle (1.pt);
\draw[color=qqqqff] (0.01,-0.83) node[anchor=south east] {$\overline{\beta}$};
\draw[color=zzqqqq] (0.15,-0.63) node {$W_0$};
\draw [color=zzqqqq,domain=0.0:1.9415421334702613] plot(\x,{(0*\x)/0.48});
\draw[color=zzqqqq] (1.42,0) node[anchor=north] {$\beta=0$};
\fill [color=qqqqcc] (0.13,-0.82) circle (1.pt);
\draw[color=qqqqcc] (0.15,-0.83) node[anchor=south east] {$C$};
\fill [color=zzqqqq] (0.17,-0.82) circle (1.pt);
\draw[color=zzqqqq] (0.15,-0.83) node[anchor=south west] {$B$};
\fill [color=zzzzqq] (0.38,-0.82) circle (1.pt);
\draw[color=zzzzqq] (0.36,-0.83) node[anchor=south west] {$A$};
\draw [dash pattern=on 4pt off 4pt, domain=0.0:1.9415421334702613] plot(\x,{(-0.48-0*
\x)/0.48});
\draw (1.42,-1) node[anchor=north] {$\beta=-1$};
\draw (1.42,-.83) node[anchor=south] {$\beta=-5/6$};
\end{tiny}
\end{tikzpicture}
\caption{First numerical wall $W_0$ for $\mathbf{v}'_2$}
\label{fig:firstwall}
\end{figure}
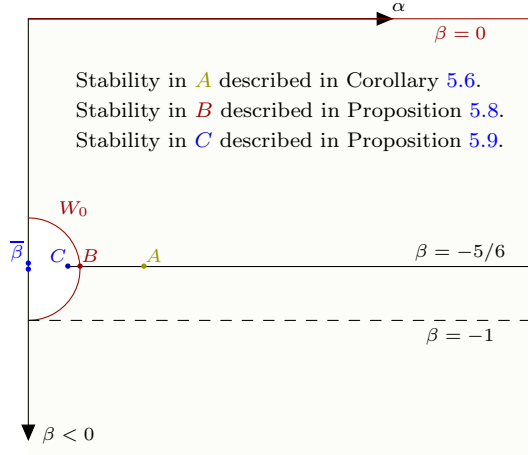

\subsection{The blow-up as a wall crossing}\label{subsec:birtransf}
Let $Y$ be a very general cubic fourfold.
Consider the Chern character $\mathbf{v}_2 = (3,0,-H^2,0,\frac{1}{4}H^4)$.
We can now complete the proof of Part~\eqref{item:main2b} of the Main Theorem.

We denote by $\mathfrak{M}_{\alpha,\beta}(\mathbf{v}_2)$ the moduli space of $\nu_{\alpha,\beta}$-semistable objects in $\Coh^\beta(Y)$ with Chern character equal to $\mathbf{v}_2$.
Since $\mathbf{v}_2' = (3,0,-H^2)$ is a primitive vector, if $(\alpha,\beta)$ is not on a wall for $\mathbf{v}_2'$, all objects are $\nu_{\alpha,\beta}$-stable, and so simple.
The moduli space of simple objects in $\Db(Y)$ is an algebraic space locally of finite type over $\CC$, by \cite{I1}.
By the main result in \cite{PT} (extended to tilt-stability for fourfolds with an analogous proof following \cite[Sections~4.5 \& 4.7]{PT}), if $(\alpha,\beta)$ is not on a wall for $\mathbf{v}_2'$, then $\mathfrak{M}_{\alpha,\beta}(\mathbf{v}_2)$ is an open subset of the moduli space of simple objects in $\Db(Y)$ and it is of finite type, separated, and proper over $\CC$.

By Part \eqref{item:main2a} of the Main Theorem and Corollary~\ref{cor:precisemoduli}, for $(\alpha,\beta)$ in the region outside the semi-circle $W_0$ and such that $\beta<0$, we have that $Z'(Y)$ is isomorphic to an irreducible component of $\mathfrak{M}_{\alpha,\beta}(\mathbf{v}_2)$, thus proving the first part of Part \eqref{item:main2b}.

Let $\mathcal{F}$ be a quasi-universal family on $Y\times Z'(Y)$, both as Gieseker-stable sheaves and as $\nu_{\alpha,\beta}$-stable objects, for $(\alpha,\beta)$ in the region outside the wall $W_0$ (see, for example, \cite[Appendix 2]{Mukai:K3} for the definition and existence of a quasi-universal family).
By \cite[Theorem~6.4]{Ku}, we have a semiorthogonal decomposition
\begin{equation}\label{eqn:quasi-universal}
\Db(Y \times Z'(Y)) = \langle \mathbf{T}_{Y \times Z'(Y)}, \O_Y\boxtimes\Db(Z'(Y)), \O_Y(H)\boxtimes\Db(Z'(Y)), \O_Y(2H)\boxtimes\Db(Z'(Y)) \rangle.
\end{equation}
Consider the relative projection $\mathcal{F}'$ of $\mathcal{F}$ on $\mathbf{T}_{Y \times Z'(Y)}$.
By Proposition~\ref{prop:tiltstabilityFCafterwall}, $\mathcal{F}'$ gives a quasi-universal family of $\nu_{\alpha,\beta}$-stable objects for $\beta=-5/6$ and $\alpha = 1/6-\epsilon$.
Let $\mathfrak{Z}$ denote the irreducible component of the moduli space $\mathfrak{M}_{1/6-\epsilon,-5/6}(\mathbf{v}_2)$ containing the sheaves $F_\Gm$ for an aCM twisted cubic $C$.
The family $\mathcal{F}'$ induces a morphism $a'\colon Z'(Y)\to \mathfrak{Z}$, which is birational on the locus corresponding to such sheaves $F_\Gm$.
Moreover, this also shows that the objects $F_\Gm'$ of Proposition~\ref{prop:tiltstabilityFCafterwall} also lie in the same irreducible component.
But
\[
\dim \Hom(F_\Gm',F_\Gm'[i]) = \begin{cases} 1,\text{ if } i=0,2\\ 8, \text{ if } i=1\\ 0, \text{ otherwise.} \end{cases}
\]
Hence, $\mathfrak{Z}$ is smooth and therefore a connected component of $\mathfrak{M}_{1/6-\epsilon,-5/6}(\mathbf{v}_2)$.

To finish the proof of Part~\eqref{item:main2b} of the Main Theorem, we only need to show that the map $a'\colon Z'(Y)\to \mathfrak{Z}$ is the contraction $b\colon Z'(Y)\to Z(Y)$ described in \cite[Section~4.5]{LLSvS}.
By \cite[Theorem~2]{Luo} (the same argument has been used in \cite[Theorem~3.10]{LMS1}), we only have to show that $a'$ contracts the same locus as the morphism $b$.
Namely, by Proposition~\ref{prop:tiltstabilityFC} and Proposition~\ref{prop:tiltstabilityFCafterwall}, this amounts to showing that for two non-CM curves $C_1$ and $C_2$, we have $F_{C_1}'\cong F_{C_2}'$ if and only if $C_1$ and $C_2$ are in the same fiber of $b$.
But \cite[Lemma~1 \& Proposition~2]{AL} exactly say what we want\footnote{To be precise, in \cite{AL}, the authors consider the Kuznetsov component $\cat{T}'_Y$ in \eqref{eqn:so2}, which is equivalent to $\cat{T}_Y$.}.

\section{$Z$ is generically a moduli space of Bridgeland-stable objects}\label{sec:Bridgelandstable}

In this section we prove Part~\eqref{item:main2c} of the Main Theorem.
The proof is based on the results in the previous section and on the construction in \cite{BLMS}. In fact, we will prove that the objects $F_C'$ are also Bridgeland-stable for a very general cubic fourfold.

\subsection{Bridgeland stability}\label{subsec:BridgelandStability}
In this section we give a brief recall on Bridgeland stability on the Kuznetsov component of a very general cubic fourfold.
We start by recalling a few results from \cite[Section~2]{AT} (see also \cite[Proposition and Definition~9.6]{BLMS}) on the Mukai structure for the Kuznetsov component.

Let the numerical K-group be the quotient of the Grothendieck group $K(\cat{T}_Y)$ by the kernel of the Euler characteristic pairing, $\chi(E, F ) = \sum_i (-1)^i \dim \Ext^i (E, F )$.
For a very general cubic fourfold we denote the numerical K-group of $\cat{T}_Y$ by $\Lambda$. 
We denote surjection $\mathbf{v}\colon K(\cat{T}_Y)\twoheadrightarrow\Lambda$ is called the Mukai vector.
Then $\Lambda$ has rank two and is generated by the classes 
\[
{\llambda}_1 := \mathbf{v}([\mathop{\mathrm{pr}}\O_L(H)])\quad \text{ and } \quad
{\llambda}_2 := \mathbf{v}([\mathop{\mathrm{pr}}\O_L(2H)]),
\]
where $L\subset Y$ denotes a line and 
$\mathop{\mathrm{pr}}:\Db(Y)\to \cat{T}_Y$ is the projection functor.
The bilinear symmetric even form $(-,-)=-\chi(-,-)$ is called the \emph{Mukai pairing}. In view of \cite[Section 2.4]{AT}, the intersection matrix of the Mukai pairing with respect to the generators ${\llambda}_1$ and ${\llambda}_2$ is:
\[
\left(\begin{array}{cc} 2 &-1\\ -1&2 \end{array}\right).
\]

\begin{defn}\label{def:stability}
Let $Y$ be a very general cubic fourfold.
A {\em Bridgeland stability condition} on $\cat{T}_Y$ with respect to the lattice $\Lambda$ is a pair $\sigma= (Z,\P)$ consisting of 
group homomorphism $Z\colon \Lambda\to\CC$ of maximal rank
and full additive subcategories $\P(\phi) \subset \cat{T}_Y$
for each $\phi\in\RR$, satisfying:
\begin{enumerate}[{\rm (a)}]
\item if $0\neq E \in \P(\phi)$, then $Z(E)\in \RR_{>0}\,\exp(i\pi\phi)$;\footnote{ 
By abuse of notation, we write $Z(E)$ instead of $Z(\mathbf{v}([E]))$ and $\mathbf{v}(E)$ instead of $\mathbf{v}([E])$ for any $E \in \cat{T}_Y$.}
\item for all $\phi\in\RR$, $\P(\phi+1)=\P(\phi)[1]$;
\item if $\phi_1>\phi_2$ and $E_j\in \P(\phi_j)$, then $\Hom_{\cat{T}_Y}(E_1,E_2)=0$;
\item (HN-filtrations) for every nonzero $E \in \cat{T}_Y$ there exists a finite sequence of morphisms
\[ 0 = E_0 \xrightarrow{s_1} E_1 \to \dots \xrightarrow{s_m} E_m = E \]
such that the cone of $s_i$ is in $\P(\phi_i)$ for some sequence
$\phi_1 > \phi_2 > \dots > \phi_m$ of real numbers.
\end{enumerate}
\end{defn}

The nonzero objects of $\P(\phi)$ are said to be \emph{$\sigma$-semistable} of \emph{phase} $\phi$, and the
simple objects of $\P(\phi)$ are said to be \emph{$\sigma$-stable}.
Note that, since $\Lambda$ has rank two and $Z$ has maximal rank, any stability condition on $\cat{T}_Y$ with respect to the lattice $\Lambda$ satisfies automatically the \emph{support property} (see, for example, \cite[Remark~2.6]{BLMS}).
In particular, Jordan-H\"older filtrations exist as well.

We denote by $\Stab(\cat{T}_Y)$ the space of stability conditions on $\cat{T}_Y$ with respect to the lattice $\Lambda$; we know that $\Stab(\cat{T}_Y)\neq \emptyset$ thanks to \cite[Theorem~1.2]{BLMS} and it has the structure of complex manifold of dimension two, by \cite[Corollary~1.3]{Bridgeland:Stab}.
We denote by $\overline{\sigma}=(\overline{Z},\overline{\P})$ the Bridgeland stability condition constructed in \cite[Section 9]{BLMS}. For the purposes of this paper, we do not need an explicit description of $\overline{\sigma}$. The only key feature is the ordering of the phases with respect to $\overline{\sigma}$ of two special objects related to lines in $Y$ which is summarized in \eqref{eqn:phasesord}.

The moduli space $\MMM^{\mathrm{spl}}(\cat{T}_Y)$ of simple objects in $\cat{T}_Y$ is an algebraic space locally of finite type over $\CC$.
Indeed, as remarked in Section~\ref{subsec:birtransf}, the corresponding statement is true for simple objects in $\Db(Y)$ by \cite{I1}. Moreover, belonging to the Kuznetsov component $\cat{T}_Y$ is an open condition for an object in $\Db(Y)$, by semicontinuity.

Let us fix a primitive non-zero Mukai vector $\mathbf{v}_0\in\Lambda$.
Let $\sigma=(Z,\P)\in\Stab(\cat{T}_Y)$ and let $\phi_0\in\RR$ be such that $Z(\mathbf{v}_0)\in \RR_{>0}\,\exp(i\pi\phi_0)$.
We denote by $\MMM_{\sigma}(\mathbf{v}_0,\phi_0)\subset\MMM^{\mathrm{spl}}(\cat{T}_Y)$ the subset parameterizing $\sigma$-stable objects in $\cat{T}_Y$ with Mukai vector $\mathbf{v}_0$ and phase $\phi_0$.
Since $\mathbf{v}_0$ is primitive and the rank of $\Lambda$ is two, all $\sigma$-semistable objects are $\sigma$-stable.
It is expected that $\MMM_{\sigma}(\mathbf{v}_0,\phi_0)$ is an algebraic space of finite type, separated, and proper over $\CC$.
While this is not available yet, it is not needed for our purposes, since the following result holds:

\begin{prop}[{\cite[Proposition~A.7]{BLMS}}]\label{prop:MukaiConnectedness}
Under the above assumptions, assume that there exists a smooth integral projective variety $M\subset \MMM_{\sigma}(\mathbf{v}_0,\phi_0)$ of dimension $\mathbf{v}_0^2+2$ such that the inclusion $M \hookrightarrow \MMM^{\mathrm{spl}}(\cat{T}_Y)$ is an algebraic morphism.
Then $M = \MMM_{\sigma}(\mathbf{v}_0,\phi_0)$ as algebraic spaces\footnote{To be precise, the functor associated to $\MMM_{\sigma}(\mathbf{v}_0,\phi_0)$ is coarsely represented by $M$.}.
\end{prop}

Note that there exists a quasi-universal family on $M$ which is actually contained in the semiorthogonal component $T_{Y \times M}$, where we use the notation as in \eqref{eqn:quasi-universal}. 
We will apply Proposition~\ref{prop:MukaiConnectedness} when $M$ is the IHS eightfold $Z(Y)$, and so we will still use the existence of $Z(Y)$ and its projectivity (namely, the main result in \cite{LLSvS}).
Having good properties for $\MMM_{\sigma}(\mathbf{v}_0,\phi_0)$, as we have in the case of K3 surfaces by \cite{BM:projectivity}, would imply the existence of $Z(Y)$ without assuming \cite{LLSvS}.

\subsection{Stability of special objects}\label{subsec:BridgelandStableObjects}
Let $Y$ be a very general cubic fourfold. 
Let $F_\Gm'\in\cat{T}_Y$ be the objects defined in Proposition~\ref{prop:tiltstabilityFCafterwall}.
To prove Bridgeland stability of $F_\Gm'$ we use a very similar argument as in \cite[Section~A.1]{BLMS}.
As in \emph{loc.~cit.}, the key ingredient is the following observation (see \cite{Mukai:K3} and \cite[Lemma~2.5]{BayerBridgeland:K3}).

\begin{lem}[Mukai]\label{lem:Mukai}
Let $A\to E \to B$ be an exact triangle in $\cat{T}_Y$.
Assume that $\Hom(A,B)=0$.
Then
\[
\dim \Ext^1(A,A) + \dim \Ext^1(B,B) \leq \dim \Ext^1(E,E).
\]
\end{lem}

The above lemma is used in the proof of the following results.
In Proposition~\ref{prop:NoSphericalNorSemirigid} though, we directly refer to \cite[Appendix~A]{BLMS}.

\begin{prop}\label{prop:NoSphericalNorSemirigid}
Let $Y$ be a very general cubic fourfold.
Then there exists no nonzero object $E\in\cat{T}_Y$ with $\Ext^1(E,E)=0$ or $\Ext^1(E,E)\cong\CC^2$.
\end{prop}

\begin{proof}
The fact that there exist no objects with $\Ext^1(E,E)=0$ is \cite[Lemma~A.2]{BLMS}.
For objects with $\Ext^1(E,E)\cong\CC^2$, we can use \cite[Lemma~A.3]{BLMS}: if such an object $E$ exists, then $\mathbf{v}(E)^2=0$, which is impossible in the lattice $\Lambda$, unless $\mathbf{v}(E)=0$.
\end{proof}

As a consequence of Proposition~\ref{prop:NoSphericalNorSemirigid}, we obtain stability of all objects with $\Ext^1\cong\CC^4$.

\begin{prop}\label{prop:StabilityFano}
Let $Y$ be a very general cubic fourfold and let $E\in\cat{T}_Y$ be an object with $\Ext^1(E,E)\cong\CC^4$.
Then, for all $\sigma\in\mathrm{Stab}(\cat{T}_Y)$, $E$ is $\sigma$-stable.
In particular, $\mathbf{v}(E)^2=2$.
\end{prop}

\begin{proof}
Let $E\in\cat{T}_Y$ be an object as in the statement and let $\sigma\in\mathrm{Stab}(\cat{T}_Y)$.
By Lemma~\ref{lem:Mukai} and Proposition~\ref{prop:NoSphericalNorSemirigid}, we can assume that $E$ is $\sigma$-semistable and that it has a unique $\sigma$-stable factor $E_0$.
If $E\neq E_0$, then $d:=\dim\Hom(E,E)\geq 2$. 
Therefore, $\mathbf{v}(E)^2= 4-2d\leq 0$, which is impossible in the lattice $\Lambda$.
\end{proof}

Finally, we can study stability of the objects $F_C'$:

\begin{prop}\label{prop:StabilityFCprime}
Let $Y$ be a very general cubic fourfold and let $E\in\cat{T}_Y$ be an object with $\Ext^{<0}(E,E)=0$, $\Hom(E,E)\cong\CC$, and $\Ext^1(E,E)\cong\CC^8$.
If $E$ is not $\sigma$-stable for some $\sigma\in\mathrm{Stab}(\cat{T}_Y)$, then its HN-filtration is of the form
\[
A \to E \to B,
\]
where $A$ and $B$ are $\sigma$-stable with $\mathbf{v}(A)^2=\mathbf{v}(B)^2=2$ and $(\mathbf{v}(A),\mathbf{v}(B))=1$.
\end{prop}

\begin{proof}
Assume that $E$ is not $\sigma$-stable.
Since $\Hom(E,E)\cong\CC$, $E$ cannot have a unique $\sigma$-stable factor.
Hence, we have a non-trivial exact triangle $A\to E \to B$ with $\Hom(A,B)=0$, $B\in \P(\phi)$, and such that the HN-factors of $A$ have phases $\geq \phi$.
By Lemma~\ref{lem:Mukai} and Proposition~\ref{prop:NoSphericalNorSemirigid}, we must have $$\Ext^1(A,A)\cong\Ext^1(B,B)\cong\CC^4.$$
By Proposition~\ref{prop:StabilityFano}, this implies that both $A$ and $B$ are $\sigma$-stable, and so it is exactly the HN-filtration of $E$.
Moreover, $\mathbf{v}(A)^2=\mathbf{v}(B)^2=2$.
Since by assumption $\mathbf{v}(E)^2=6$, this gives $(\mathbf{v}(A),\mathbf{v}(B))=1$.
\end{proof}

\subsection{Proof of Part~\eqref{item:main2c} of the Main Theorem}\label{subsec:ProofOfPart2c}

Let $Y$ be a very general cubic fourfold, and let $\overline{\sigma}=(\overline{Z},\overline{\P})$ be any of the Bridgeland stability conditions on $\cat{T}_Y$ constructed in \cite[Section~9]{BLMS}.
Part~\eqref{item:main2c} of the Main Theorem follows immediately from the following result.

\begin{prop}\label{prop:BridgelandStabilityFC'}
For any generalized twisted cubic curve $\Gm$, the object $F_{\Gm}'$ is $\overline{\sigma}$-stable.
\end{prop}

Indeed, by Proposition~\ref{prop:BridgelandStabilityFC'} and by Part \eqref{item:main2b} of the Main Theorem, the IHS eightfold $Z(Y)$ is contained in the moduli space $\MMM_{\overline{\sigma}}(\mathbf{v}_0,\phi_0)$, where $\mathbf{v}_0=\mathbf{v}(F_{\Gm}')$ and $\phi_0$ is chosen accordingly, and the inclusion $Z(Y) \hookrightarrow \MMM^{\mathrm{spl}}(\cat{T}_Y)$ is an algebraic morphism.
The theorem then follows directly from Proposition~\ref{prop:MukaiConnectedness}.

To prove Proposition~\ref{prop:BridgelandStabilityFC'} we will use Proposition~\ref{prop:StabilityFCprime}.
To this end, the first thing is to compute the numerical class $\mathbf{v}_0$.

\begin{lem}\label{lem:ClassOfFCprime}
We have $\mathbf{v}_0=\mathbf{v}(F_{\Gm}')=2\llambda_1 + \llambda_2 \in \Lambda$.
\end{lem}

\begin{proof}
This is a straightforward computation, by using the definitions.
\end{proof}

By Proposition~\ref{prop:StabilityFCprime} and Lemma~\ref{lem:ClassOfFCprime}, if we assume that $F_{\Gm}'$ is not $\overline{\sigma}$-stable, then we must have $\mathbf{v}(A),\mathbf{v}(B)\in\{\llambda_1,\llambda_1+\llambda_2 \}$.
Stable objects with class $\llambda_1$ can be easily determined, by using Proposition~\ref{prop:StabilityFano}.
Indeed, we recall the following construction of Kuznetsov--Markushevich from \cite{KM}, and used in \cite[Appendix~A]{BLMS}.
Given a line $L\subset Y$, we define a torsion-free sheaf $F_L$ as the kernel of the (surjective) evaluation map
\[
F_L:= \ker \left( H^0(Y,\I_{L/Y}(H))\otimes\mathcal{O}_Y \twoheadrightarrow \I_{L/Y}(H)\right),
\]
where $\I_{L/Y}$ is the ideal sheaf of $L$ in $Y$. Then by \cite[Section 5]{KM}, $F_L$ is a torsion-free Gieseker-stable sheaf on $Y$ which has the same $\Ext$-groups as $\I_{L/Y}$, and which belongs to $\cat{T}_Y$.
Its Chern character is $\ch(F_L)=(3,-H,-\frac{1}{2} H^2, \frac{1}{6}H^3,\frac{1}{8}H^4)$.
By definition of $\llambda_1$, one also easily verifies that $\mathbf{v}(F_L)=\llambda_1$.
By letting $L$ vary, the sheaves $F_L$ span a connected component of the moduli space of Gieseker-stable sheaves which is isomorphic to the Fano variety of lines $F(Y)$ \cite[Proposition~5.5]{KM}.
By Proposition~\ref{prop:StabilityFano}, $F_L$ is $\overline{\sigma}$-stable, and by Proposition~\ref{prop:MukaiConnectedness} the moduli space $\MMM_{\overline{\sigma}}(\llambda_1,\phi(F_L))$ is isomorphic to $F(Y)$.

Bridgeland stable objects with Mukai vector $\llambda_1+\llambda_2$ can also be easily described.
Indeed, given a line $L\subset Y$, we can define objects $P_L$ as in \cite[Section 2.3]{MS1}.
They are non-trivial extensions
\[
\O_Y(-H)[1] \to P_L \to \I_{L/Y},
\]
defined as the right mutation $P_L:=\cat{R}_{\O_Y(-H)}(F_L(-H))$.
In particular, $P_L$ is obtained from $F_L$ by applying an autoequivalence of $\cat{T}_Y$, and so it has the same $\Ext$-groups as $F_L$ and it is easy to check that its class in $\Lambda$ is exactly $$\mathbf{v}(P_L)=\llambda_1+\llambda_2.$$
Therefore, they are also all $\overline{\sigma}$-stable, and their moduli space $\MMM_{\overline{\sigma}}(\llambda_1+\llambda_2,\phi(P_L))$ is isomorphic to $F(Y)$ as well.

We can also easily compute the corresponding phases of $F_L$ and $P_L$ with respect to $\overline{\sigma}$.
To this end we use the computation in the proof of \cite[Proposition~9.10]{BLMS}; we have
\begin{equation}\label{eqn:phasesord}
\phi(F_L)<\phi(P_L)<\phi(F_L)+1.
\end{equation}
Hence, by Proposition~\ref{prop:StabilityFCprime}, if $F_{\Gm}'$ is not $\overline{\sigma}$-stable, we must have an exact triangle
\[
P_L \to F_{\Gm}' \to F_L,
\]
namely, in the notation of the proposition, we must have $A=P_L$ and $B=F_L$.
Hence, to prove Proposition~\ref{prop:BridgelandStabilityFC'}, it is enough to show that $\Hom(F_{\Gm}',F_L^{})=0$.
This may be checked directly, or it follows immediately from Lemma~\ref{lem:FLstable} below, thus concluding the proof of Proposition~\ref{prop:BridgelandStabilityFC'}, and so the proof of the Main Theorem.

\begin{lem}\label{lem:FLstable}
The sheaf $F_L$ is $\nu_{\alpha,\beta}$-stable, for all $\alpha>0$ and $\beta<-1/3$.
Moreover, we have $\nu_{\alpha,\beta}(F_L^{})<\nu_{\alpha,\beta}(F_{\Gm}')$ for $(\alpha,\beta)\in W_0$.
\end{lem}

\begin{proof}
This can be shown, similarly to Lemma~\ref{lem:firstwall}, by looking at the semi-line $\beta_0=-1$, and letting $\alpha$ vary.
Indeed, the unique straight wall corresponds to $\beta=-1/3$ and we are interested in the region $\beta<-1/3$ so that the objects $F_L$ are in $\cat{Coh}^\beta (Y)$.
Moreover, the point of accumulation for the center of all walls for $F_L$ in this region is $\overline{\beta}=-1$.
Hence, if we prove that $F_L$ is $\nu_{\alpha,\beta_0}$-stable, for all $\alpha>0$, we have that it is $\nu_{\alpha,\beta}$-stable for all $\alpha>0$ and $\beta<-1/3$. The inequality between the slopes is then a direct computation.

For $\alpha\gg0$ large, since $(3,-H,-H^2/2)$ is primitive and $-1/3>-1$, the objects $F_L$ are $\nu_{\alpha,\beta_0}$-stable.
Then, as in Lemma~\ref{lem:firstwall}, we can check by computer that there is no wall intersecting the semi-line $\beta_0=-1$. 
\end{proof}


\medskip

{\small\noindent{\bf Acknowledgements.} Parts of this paper were written while Paolo Stellari was visiting the Department of Mathematics of the Ohio State University, Northeastern University, and the Universit\'e Paris Diderot -- Paris~7.
The authors were working on this project while Mart\'i Lahoz, Emanuele Macr\`i and Paolo Stellari were visiting the Institute of Mathematics of the University of Bonn.
During the revision of the paper, Emanuele Macr\`i was holding a Poincar\'e Chair from the Institut Henri Poincar\'e and the Clay Mathematics Institute.
The warm hospitality and the financial support of these institutions are gratefully acknowledged.
It is a pleasure to thank Nick Addington, Christian Lehn, and Benjamin Schmidt for very useful conversations on a preliminary version of this paper and for their comments on its first draft.
We are very grateful to Evgeny Shinder and Andrey Soldatenkov for generously sharing with us a preliminary version of \cite{SS}.
The paper benefited very much from the careful reading and insightful comments and suggestions from the referee, whom we also thank for encouraging us to improve the exposition in Section~\ref{sec:tiltstable} and to add Section~\ref{sec:Bridgelandstable}.
}


\end{document}